\documentclass[10pt]{amsart}
\usepackage{cite,mathrsfs,amsmath,amssymb,amsthm,bbm,esint,enumerate,stmaryrd}

\usepackage[breaklinks=true]{hyperref}

\newtheorem{lem}{Lemma}[section]
\newtheorem{thrm}[lem]{Theorem}
\newtheorem{prop}[lem]{Proposition}
\newtheorem{cor}[lem]{Corollary}
\theoremstyle{definition}
\newtheorem{defn}[lem]{Definition}

\theoremstyle{remark}

\newcommand{\eq}[2]{\begin{equation}\label{#1}#2\end{equation}}

\newcommand{\R}{\mathbb{R}}
\newcommand{\C}{\mathbb{C}}
\newcommand{\Z}{\mathbb{Z}}
\newcommand{\N}{\mathbb{N}}

\newcommand{\Schwartz}{\mathcal{S}}
\newcommand{\Test}{C_c^\infty}

\renewcommand{\Im}{\operatorname{Im}}
\renewcommand{\epsilon}{\varepsilon}
\newcommand{\mc}{\mathcal}

\newcommand{\mf}{\mathfrak}
\newcommand{\mbb}{\mathbb}

\DeclareMathOperator{\supp}{supp}

\DeclareMathOperator{\bbE}{\mathbb E}
\newcommand{\<}{\langle}
\renewcommand{\>}{\rangle}
\newcommand{\p}{\partial}

\newcommand{\bbo}{\mathbbm 1}

\newcommand{\bpf}{\begin{proof}}
\newcommand{\epf}{\end{proof}}
\newcommand{\qtq}[1]{\quad\text{#1}\quad}

\newcommand{\qt}[1]{\quad\text{#1}}

\newcommand{\LHS}[1]{\text{LHS}\eqref{#1}}
\newcommand{\RHS}[1]{\text{RHS}\eqref{#1}}

\newcommand{\cN}{\mc N}
\newcommand{\II}{I\!I}
\newcommand{\III}{I\!I\!I}
\newcommand{\bbP}{\mbb P}
\newcommand{\X}{{L_\mu^2}}
\newcommand{\Xe}{{L_{\mu^\epsilon}^2}}

\newcommand{\dmue}{\tfrac{d\mu^{\mkern-1mu\epsilon}\!\!}{dx}\,}
\newcommand{\dmut}{\tfrac{d\mu^{\mkern-1mu\eta}\!\!}{dx}\,}

\title[NLS with sprinkled nonlinearity]{The nonlinear Schr\"odinger equation with sprinkled nonlinearity}
\author{Benjamin Harrop-Griffiths}
\address{Department of Mathematics \& Statistics, Georgetown University\\Washington, DC 20057, USA}
\email{benjamin.harropgriffiths@georgetown.edu}
\thanks{}

\author{Rowan Killip}
\address{Department of Mathematics, University of California, Los Angeles\\ CA 90095, USA}
\email{killip@math.ucla.edu}
\thanks{}

\author{Monica Vi\c san}
\address{Department of Mathematics, University of California, Los Angeles\\CA 90095, USA}
\email{visan@math.ucla.edu}
\thanks{}

\numberwithin{equation}{section}
\allowdisplaybreaks

\begin{document}

\begin{abstract}
We prove global well-posedness for the cubic nonlinear Schr\"odinger equation with nonlinearity concentrated on a homogeneous Poisson process.
\end{abstract}

\maketitle

\section{Introduction}

We consider the dynamics of a field $\psi:\R\to\C$ under the Hamiltonian
\begin{align}\label{Emu}
E[\psi] := \int_\R \tfrac12|\psi'(x)|^2 \,dx + \int_\R \tfrac12|\psi(x)|^4\,d\mu(x),
\end{align}
where $\mu$ is a measure on $\R$. This leads to the nonlinear Schr\"odinger equation
\begin{align}\label{NLS}\tag{NLS${}_\mu$}
i\psi_t = -\psi_{xx} + 2|\psi|^2\psi \,d\mu.
\end{align}

This model is relevant in physical scenarios where the wave self-interaction is mediated by the environment.   The measure $\mu$ captures the strength of these effects and how they vary in space due to variations in the ambient matter.   In optics, for example, nonlinear effects originate from the interaction of light with the atoms of the matter through which it is traveling \cite{Boyd}. 

As the atomic scale is significantly smaller than the wavelength of visible light, one may hope that homogenization effects allow one to treat the light-matter interaction as spatially homogeneous and so take $d\mu$ as a multiple of Lebesgue measure.  Indeed, one may interpret the evident successes in modeling optical experiments with this homogeneous NLS (see \cite{MR3308230}) as compelling evidence that the presumption of homogenization was justified.  

The quest to better understand such homogenization phenomena is one of our motivations for introducing the sprinkled nonlinearity model in this paper.  In this model,  $\mu$ is random and distributed according to a Poisson process (we will elaborate on this later).   We contend that this model is physically natural; moreover, it is also consciously adversarial to naive notions of homogenization.  Nevertheless, the question of homogenization of the sprinkled nonlinearity has been successfully treated in the subsequent and very recent paper \cite{HGN}.

The most obvious anathema to the traditional homogeneous NLS (in which $\mu$ is Lebesgue measure) is to choose $\mu$ as a single point mass.  A substantial literature on this model has been developed in both mathematics and physics.  For a thorough account of the mathematical literature on this point-concentrated nonlinearity, we recommend the recent review \cite{MR4653819}.   Such highly concentrated nonlinearities arise as effective models in a number of physical situations, as described in  
\cite{F_Kh_Abdullaev_2004, PhysRevB.56.15090, PhysRevE.84.056609, LIDORIKIS1998346, PhysRevB.54.1537, PhysRevB.47.10402, 10.1119/1.1417529}.  Indeed, advances in nanoscale fabrication have led researchers to consider double barriers \cite{JONALASINIO19951}, as well as periodic arrays of such nonlinearities \cite{PhysRevB.44.13082, 10.1063/1.111416, PhysRevA.101.043830, SENOUCI2014592, KSenouci1999, PhysRevE.65.036609, PhysRevB.51.11221}.  

Rather than placing the centers of nonlinear interaction in a regular array, we are interested in sprinkling them in a completely random manner through the ambient material.  The simplest physical model for such a sprinkling of points is the Gibbs state for a gas of non-interacting particles.  The resulting distribution of points is known as a (homogeneous) Poisson process. This is also the natural model for dilute impurities in a crystal (or raisins in a plum pudding).  We will discuss these connections more fully in subsection~\ref{SS:PP}.

Given the locations for our point interactions, we build the measure $\mu$ by placing a delta measure at each point.  As there will only be finitely many points in each compact set, our measure $\mu$ is locally finite, that is, it gives finite mass to any compact set. We write $\mf M(\R)$ for the set of locally finite Borel measures on $\R$.  We make $\mf M(\R)$ a measurable space by endowing it with the Borel $\sigma$-algebra $\mf B$ induced by the vague topology, that is, the weakest topology under which $\mu \mapsto \int \phi \,d\mu$ is continuous for every continuous function $\phi:\R\to\C$ of compact support.

To describe a random configuration of point interactions, we must allow $\mu$ to be random.  This means choosing a probability space $(\Omega,\mf F,\bbP)$ and a (measurable) map $\mu:\Omega\to\mf M(\R)$.  In this setup, the (random) number of points in a compact set $K\subset\R$ is given by the random variable $\mu(K):\Omega\to[0,\infty)$ defined by $\omega\mapsto \mu(K;\omega)$.  Similarly, the expected (=average) number of points in the set $K$ may be written variously as
$$
\bbE\Bigl\{ \mu(K) \Bigr\} = \int_\Omega \mu(K;\omega)\, d\bbP(\omega)
	= \int_\Omega \int_\R \! \bbo_{\mkern-2mu K}(x) \,d\mu(x;\omega)\, d\bbP(\omega)
	= \bbE\Bigl\{ \int_\R\! \! \bbo_{\mkern-2mu K}(x) \,d\mu(x)\Bigr\} .
$$
In general, we will suppress the dependence on $\omega$ as in the first and last expressions.

The specific distribution of points we wish to consider is the homogeneous Poisson process:

\begin{defn}\label{D:pp}
A random measure $\mu:\Omega\to\mf M(\R)$ is (distributed as) the \emph{homogeneous Poisson process of unit intensity} if it satisfies
\begin{equation}\label{E:pp}
\bbE\Bigl\{ \exp\Bigl(-\textstyle\int_\R \phi(x) \,d\mu(x)\Bigr) \Bigr\}
	= \exp\Bigl(-\textstyle\int_\R [1-e^{-\phi(x)}] \,dx\Bigr)
	\qtq{for all} \phi\in C_c(\R).
\end{equation}
\end{defn}

The functionals appearing on LHS\eqref{E:pp} are known as Laplace functionals.  They uniquely determine the distribution of the measure $\mu$; see \cite[Corollary~2.3]{MR854102}.

This is a \emph{Poisson} process because the number of points $\mu(K)$ in any compact set $K$ follows a Poisson distribution.  The average number of points is $|K|$,  the Lebesgue measure of $K$.  More generally, for any collection of disjoint measurable sets $E_j$, the random variables $\mu(E_j)$ are \emph{independent} Poisson random variables with means $|E_j|$.

The special role of Lebesgue measure in dictating the average number of points in any set makes our Poisson process \emph{homogeneous} and of \emph{unit intensity}. Homogeneity expresses the translation invariance of Lebesgue measure.  One may alter the intensity by multiplying Lebesgue measure by a (positive) numerical factor.  By employing scaling, however, we see that there is no loss of generality in selecting unit intensity.   We have chosen the process to be homogeneous because we regard this as the most physical case.  Nevertheless, our methods allow much flexibility in this regard, as will become evident in due course. 

The main result of this paper is the demonstration of (almost-sure) well-posedness for the nonlinear Schr\"odinger equation \eqref{NLS} when the nonlinearity is concentrated on a Poisson process:

\begin{thrm}\label{t:main}Let \(\mu\) be distributed according to the Poisson process \eqref{E:pp}.  Then the equation \eqref{NLS} is almost surely globally well-posed in \(H^1(\R)\).  Specifically,\\
{\upshape(a)} Given any initial data \(\psi_0\in H^1\), almost surely there is a unique \(\psi\in C(\R;H^1)\) that satisfies
\eq{Duhamel NLS}{
\psi(t) = e^{it\p_x^2}\psi_0 - 2i\int_0^te^{i(t-\tau)\p_x^2}\Bigl[|\psi(\tau)|^2\psi(\tau)d\mu\Bigr]\,d\tau.
}
{\upshape(b)} The energy \eqref{Emu} of the solution is both finite and conserved; so too is the mass
\begin{equation}\label{Mmu}
M[\psi] = \int_\R |\psi|^2\,dx.
\end{equation}
{\upshape(c)}  For any sequence of initial data, \(\psi_{n,0}\to \psi_0\) in \(H^1\), the corresponding solutions \(\psi_n\in C(\R;H^1)\), which exist almost surely, converge to $\psi$ in $L^p\bigl(d\bbP;C([-T,T];H^1)\bigr)\mkern-5mu:$
\eq{continuity}{
\lim_{n\to \infty}\bbE\Bigl[\,\sup_{|t|\leq T}\|\psi_n(t) - \psi(t)\|_{H^1}^p\Bigr]= 0,
}
for any time \(T>0\) and any $1\leq p<\infty$.
\end{thrm}

We say that a \(\psi\in C(\R;H^1)\) satisfying \eqref{Duhamel NLS} is a \emph{strong solution} to \eqref{NLS} with initial data $\psi_0$.  Note that RHS\eqref{Duhamel NLS} will, almost surely, constitute a tempered distribution for any $\psi\in C(\R;H^1)$.  The crux of this is the convergence of 
\begin{equation}\label{stdual}
\int_0^t\int_{-\infty}^\infty \overline{e^{-i(t-\tau)\partial_x^2}\phi(x)} |\psi(\tau,x)|^2\psi(\tau,x)\,d\mu(x)\,d\tau
\end{equation}
for any $t\in\R$ and any $\phi\in\Schwartz(\R)$; indeed, $\psi(\tau,x)$ is uniformly bounded on the region of integration and $e^{-i(t-\tau)\partial_x^2}\phi$ belongs to Schwartz space, which ensures sufficient decay to integrate $d\mu$.

On the other hand, it is far from clear that RHS\eqref{Duhamel NLS} belongs to $H^1(\R)$, as claimed in Theorem~\ref{t:main}.  There is, of course, the question of regularity; this appears already in the case of a point nonlinearity.  The more poignant problem here is decay: For almost every $\mu$ there is a choice of $\phi\in H^1$ and $\psi\in C(\R;H^1)$ so that \eqref{stdual} diverges for every $t>0$!

The origins of such divergence can be seen more easily at the level of the energy \eqref{Emu}.  For \(k\in \Z\), the random variables \(\mu\bigl([k,k+1)\bigr)\) are independent and Poisson-distributed with unit mean. In particular, the probability that there are $n$ many points in any unit interval $[k,k+1)$ is $e^{-1}/n!$. The Borel--Cantelli lemma then ensures that, almost surely, for all \(n\geq 1\) we have \(\mu\bigl([k,k+1)\bigr)\geq n\) for infinitely many choices of $k$. As a consequence, for almost every sample $\mu$ of the Poisson process there exists a sequence of distinct integers \(x_n\) with the property that
\begin{equation}\label{rich}\mu\bigl([x_n,x_n+1)\bigr) \geq n \qtq{for each} n\in\N.
\end{equation} It is then elementary to choose a $\psi_0\in H^1(\R)$, comprised of bumps (centered on these intervals) with slowly decaying amplitudes, so that $\psi_0$ has infinite energy.  This divergence shows that the (zero probability) set of exceptional measures $\mu$ that is excluded in Theorem~\ref{t:main} must depend on $\psi_0$.

Earlier, we described highly concentrated nonlinearities as representative of the case where the characteristic scale of the nonlinearity is much much narrower than that of the wave.  This can be made rigorous by considering a family of models with mollified nonlinearities and showing that their solutions converge to those of \eqref{NLS}.  Just such a result was proved for finite combinations of point interactions in \cite{MR3275343,MR2318828}.  We wish to do the same for our model.  To this end, we define the absolutely continuous measure $\mu^\epsilon$ with density (Radon--Nikodym derivative) 
\begin{equation}\label{mueps}
\dmue(x) = \int_\R\rho^\epsilon(x-y)\,d\mu(y), \qtq{where} \rho^\epsilon(x) = \tfrac1\epsilon\rho\left(\tfrac x\epsilon\right)
\end{equation}
and \(\rho\in \Test(-1,1)\) is non-negative, even, and satisfies \(\int_\R\rho(x)\,dx = 1\).  

\begin{thrm}\label{t:demol}
Fix $\psi_0\in H^1(\R)$ and let \(\mu\) be distributed according to the Poisson process \eqref{E:pp}.  Then, almost surely, the solutions $\psi_\epsilon(t)$, associated to the  mollified measures $\mu^\epsilon$ defined in \eqref{mueps},  converge to the solution $\psi(t)$ to \eqref{NLS}; indeed,
\begin{align}\label{eps to 0}
\lim_{\epsilon\to 0}\bbE\Bigl[\,\sup_{|t|\leq T}\|\psi_\epsilon(t) - \psi(t)\|_{H^1}^p\Bigr]= 0,
\end{align}
for any time \(T>0\) and any $1\leq p<\infty$.
\end{thrm}

If $\mu$ were a finite measure, then each $\mu^\epsilon$ would have a bounded density and correspondingly, the well-posedness of the regularized model would be covered by the standard arguments described in \cite{MR2002047}.  For the Poisson process, however, the mollification procedure does nothing to reduce the impact of the sequences of  `rich' intervals such as the intervals $[x_n,x_n+1)$ described in \eqref{rich}.  In this way, the model with mollified coupling measure retains many of the original difficulties and this compels us to develop a unified well-posedness theory.  This theory will be deterministic and can be applied to very general measures $\mu$; see Section~\ref{s:deterministic}.

As we have explained, our archenemy in this work is the inevitable, but extremely erratic, appearance of rich clusters of points in samples of the Poisson process.   It is crucial to control the wave's interaction with these clusters.  For this purpose, we introduce a carefully chosen weighted $L^2$ norm that conforms to the highs and lows of the process.

Given a positive measure $\mu$ on $\R$ and an integer \(k\in \Z\) we define
\eq{Nk}{
\cN_k(\mu)^2 = 4 + \sup_{\ell\in \Z}\Bigl[\mu(I_\ell)^2 - |k - \ell|\Bigr],\qtq{where} I_\ell = [\ell-\tfrac12,\ell+\tfrac12).
}
This function provides a slowly varying approximation to \(\mu(I_k)\) in which rich intervals have a correspondingly wide range of influence.  We then define an associated weight via linear interpolation:
\eq{omega}{
w(x;\mu) = (1+k-x)\cN_k(\mu)^2 + (x-k)\cN_{k+1}(\mu)^2\qt{for \(x\in [k,k+1)\)}.
}
Finally, we define \(\X := L^2(\R;w(x;\mu)\,dx)\), that is, the Hilbert space with norm
\[
\|f\|_\X^2 = \int_\R |f(x)|^2\,w(x;\mu)\,dx.
\]

Naturally, we will only consider measures $\mu$ for which $w(x;\mu)$ is finite.  As we will show in Lemma~\ref{l:omega}, this is guaranteed by the single condition $\cN_0(\mu)<\infty$.  Lemma~\ref{l:Nk finite} then ensures that $\cN_0(\mu)$ is almost surely finite for samples from the Poisson process.

The engine of our paper is the following entirely deterministic well-posedness result: 

\begin{thrm}\label{t:deterministic}
Let \(\mu\) be a positive Borel measure for which \(\cN_0(\mu)<\infty\). Then \eqref{NLS} is globally well-posed in \(H^1\cap \X\).

Precisely, given any initial data \(\psi_0\in H^1\cap \X\) there exists a unique global strong solution \(\psi\in C(\R;H^1\cap \X)\) of \eqref{NLS} satisfying \(\psi(0) = \psi_0\). This solution conserves both the mass and the energy. Moreover, for any time \(T>0\) the solution map \(H^1\cap \X\ni\psi_0\mapsto \psi\in C([-T,T];H^1\cap \X)\) is continuous.
\end{thrm}

This result is used to prove Theorem~\ref{t:main} in Section~\ref{s:main proof}.  Indeed, Lemma~\ref{l:Nk finite} shows that any fixed $\psi_0\in H^1(\R)$ belongs to $L^2_\mu$ for almost every sample $\mu$ from the Poisson process \eqref{E:pp}.   The arguments used to prove Theorem~\ref{t:deterministic} also yield a deterministic analogue of Theorem~\ref{t:demol}; see Corollary~\ref{C:demol}.  This constitutes the basis for the proof of Theorem~\ref{t:demol} in Section~\ref{s:main proof}.

\subsection{Outline of the proof} 
We first focus on the deterministic result Theorem~\ref{t:deterministic}. As Lemma~\ref{l:Nk finite} demonstrates, the restriction $\cN_0(\mu)<\infty$ is satisfied by samples from the Poisson process; consequently, we must accept that there exist unit intervals with arbitrarily large $\mu$ measure.   On the other hand, the condition that $\psi_0\in L^2_\mu$ ensures that initially, at least, the waveform is not strongly concentrated on such rich clusters.  In particular, the energy is initially finite; see Lemma~\ref{l:measure ests}.

To proceed, we must confront the natural (and inevitable) spreading of dispersive waves into such rich regions.   Indeed, a concentrated nonlinearity naturally creates high frequency waves and these, we fear, may carry the mass to regions where $\mu$ is large.  While we believe such transportation does take place, we will ultimately be able to control it in a satisfactory manner by proving $L^2_\mu$ bounds. 

For nonlinear Schr\"odinger equations, the transportation of $L^2$ norm is governed by the mass current (= momentum density), which is given by $2\Im[\;\!\overline\psi \psi'\;\!](x)$.  This we will control using the conservation of energy.
 
Energy conservation is, of course, a property of true solutions, which is precisely what we are trying to construct!  It is not a property of Picard iterates, for example.  In fact, construction of solutions via Picard iteration is problematic: one needs $H^1$ bounds on the iterates to control transport and one needs transport bounds to control $H^1$ growth.  Moreover, there is no easy way to exhibit the smallness needed to close a bootstrap argument.  One may make the time interval small, but how short it needs to be is delicately entwined with the spreading of the Picard iterates into rich intervals.

Instead, we will construct solutions in the style of an energy method.  The first step is to regularize the model, not solely in the sense of \eqref{mueps}, but also by limiting the existence of rich clusters.  This we do with a spatial truncation:
\eq{V eps}{
V^\epsilon(x) = \varphi^\epsilon(x) \dmue(x)\in \Test(\R) \qtq{where} \varphi^\epsilon(x) = \varphi(\epsilon x)
}
and \(\varphi\in \Test(-2,2)\) is a non-negative bump function that is identically \(1\) on \([-1,1]\).
The resulting model is easily solved by contraction mapping; see Proposition~\ref{p:regularized}.

We then show that the sequence of solutions to this regularized model converge as $\epsilon\to0$, at least on short time intervals, and moreover, the limit $\psi(t,x)$ is an $L^\infty_t H^1_x$ solution to \eqref{NLS}; see Proposition~\ref{p:local existence}.  At this moment in the argument, we know that $t\mapsto\psi(t)$ is weakly $H^1$-continuous, but we do not know that it is norm continuous.   Nevertheless, Proposition~\ref{p:weak Lipschitz} guarantees uniqueness of such a solution.

To prove well-posedness of \eqref{NLS}, we must show that the limiting solution $\psi$ belongs to $C(\R; H^1)$ and demonstrate continuity of the data-to-solution map.  We start by proving that $\psi$ conserves both mass and energy.  This is then used to demonstrate that $\psi$ is a global in time solution and that it belongs to $C(\R; H^1)$; see Proposition~\ref{p:global existence}.  Finally, the proof of Theorem~\ref{t:deterministic} is completed with Proposition~\ref{p:cont}, which shows that the data-to-solution map $\psi_0\mapsto\psi(t)$ is $H^1\cap\X$-continuous.  In this way, the deterministic analysis is completed in Section~\ref{s:deterministic}.

Section~\ref{s:main proof} is devoted to proving our main results about the stochastic model, namely, Theorems~\ref{t:main} and~\ref{t:demol}.  Lemma~\ref{l:Nk finite} shows that our deterministic Theorem~\ref{t:deterministic} is almost surely applicable for each initial data $\psi_0\in H^1$.  This leaves us to verify the convergence claims \eqref{continuity} and \eqref{eps to 0}.  Our earlier results already yield almost sure convergence.  We deduce $L^p(d\bbP)$ convergence by exhibiting a suitable dominating function and applying the dominated convergence theorem.

\subsection{Summary of existing literature} To the best of our knowledge, the model discussed in this paper has not been considered previously.  It synthesizes two facets of the theory that have been considered before: the coupling constant is given by a stationary process and is concentrated on a zero-dimensional set.

Nonlinear Schr\"odinger equations with an inhomogeneous coupling constant in front of the nonlinearity arise very naturally in many physical contexts.  This has led to a great many mathematical investigations of this case, too numerous to recount here. However, the literature on this subject (that we know of) has centered on the case of a single localized perturbation of the coupling constant.  In an experimental setting, this corresponds to a carefully engineered apparatus.
By comparison, the paper \cite{MR4162949} stands out to us for treating coupling constants given by certain stationary stochastic processes, including periodic, quasiperiodic, and alloy-type models.  Concretely, \cite{MR4162949}  proves homogenization to the case of constant coupling constant in the mass-critical two-dimensional setting.

As noted earlier, nonlinear Schr\"odinger equations with nonlinearity concentrated at a single point (or at finitely many points) have generated considerable interest in both mathematics and physics; see the recent review \cite{MR4653819}.  The quantity of literature on these models is so great that we will confine our discussion to the cubic model in one dimension and focus only on well-posedness questions.

A cubic nonlinearity with a single delta measure at $x=0$ as coupling constant
\begin{gather}\label{NLSdelta}\tag{NLS${}_\delta$}
i\psi_t = -\psi_{xx} \pm 2\delta|\psi|^2\psi 
\end{gather}
is a mass-critical model. This is the observation that the scaling
$$
\psi(t,x) \mapsto \sqrt{\lambda} \psi(\lambda^2t,\lambda x)
$$
that preserves the class of solutions also preserves the (conserved) mass \eqref{Mmu}.  As we will describe, the known phenomenology of \eqref{NLSdelta} mimics that of the mass-critical homogenous equation.

Local existence for \eqref{NLSdelta} with initial data in $L^2(\R)$ is known and such solutions were shown to be global in the case of small initial data; see \cite{MR4548487}.  As observed in \cite{MR4018570}, the model \eqref{NLSdelta} is pseudoconformally invariant.  This transformation yields explicit blow-up solutions in the focusing case.

Most researchers have discussed this problem for initial data $\psi_0\in H^1$. For such data, the defocusing \eqref{NLSdelta} was shown to be globally well-posed in \cite{MR1814425,MR2318828}.  In the focusing case, there exists an explicit mass threshold so that \eqref{NLSdelta} is globally well-posed in $H^1$ below this threshold and ill-posed above it; see \cite{MR4018570}.  For either sign of the nonlinearity, small solutions scatter; see \cite[\S2]{MR4188177}.

As noted earlier in the context of Theorem~\ref{t:demol}, the papers \cite{MR3275343,MR2318828} have already demonstrated the validity of the $\delta$ interaction as describing the limit of less singular nonlinearities.

\subsection{Point processes}\label{SS:PP} The purpose of this section is to provide additional background on possible choices for the distribution of the concentrated nonlinearities. In particular, we will describe two physical origins of the Poisson process, as well as its connection to other models.

The statistics of a gas of a fixed number of particles in finite volume is described by the canonical ensemble \cite{MR289084,MR1042093}.  Let us consider $n$ non-interacting particles in the finite interval $\Lambda=[-L,L]$.  In this case, the Laplace functionals of the particle positions $X_j$ can be easily evaluated: for any $\phi \in C_c(\R)$,
\begin{equation}
\begin{aligned}\label{can Lap fun}
\bbE_{n,\Lambda}\Bigl\{ \exp\Bigl( - {\textstyle\sum_{j=1}^n } \phi(X_j) \Bigr)\Bigr\} &= \frac1{|\Lambda|^{n}} \int_{\Lambda^n} \exp\Bigl( - {\textstyle\sum_{j=1}^n } \phi(x_j) \Bigr) \,dx_1 \ldots\,dx_n \\
&= \biggl( 1 + \frac1{|\Lambda|} \int_{\Lambda}  \bigl[ e^{- \phi(x)} - 1 \bigr] \,dx \biggr)^n .
\end{aligned}
\end{equation}
This shows that the $X_j$ are just $n$ statistically independent samples from the uniform distribution on $\Lambda$.  (There is no temperature dependence because there is no interaction.)

The Poisson process \eqref{E:pp} arises naturally in the thermodynamic limit of this model. To take the thermodynamic limit, we must send both $n$ and $L$ to infinity in such a way that the density $n/|\Lambda|$ converges.  Specifically, taking this limit with $n/|\Lambda| \to 1$, we see that the Laplace functionals \eqref{can Lap fun} converge to those of the Poisson process \eqref{E:pp}.  This in turn guarantees convergence to the Poisson process \eqref{E:pp}; see \cite[\S4.2]{MR854102}.

As noted earlier, the investigation of electron and photon dynamics in periodic structures have led a number of researchers to consider the model \eqref{NLS} with
\begin{align}\label{333}
d\mu(x) = \sum_{k\in\Z} \delta(x-k) \,dx ;
\end{align}
see \cite{PhysRevB.44.13082, 10.1063/1.111416, PhysRevA.101.043830, SENOUCI2014592, KSenouci1999, PhysRevE.65.036609, PhysRevB.51.11221}.  This has come to be known as the nonlinear Kronig--Penney model.  We are not aware of any prior discussion of the well-posedness problem for this model; however, it fits easily within the scope of Theorem~\ref{t:deterministic}.  Indeed, the absence of randomness and of rich intervals (in the sense of \eqref{rich}) make this model rather easier to treat than the general case we consider in this paper.

Earlier in the introduction we mentioned that the Poisson process also arises as the natural model of dilute impurities in a crystal.  This can be understood as a variation on \eqref{333}. Specifically, consider
\begin{align}\label{444}
d\mu(x) = \sum_{k\in\Z} Z_k \delta(x-k h ) \,dx
\end{align}
where $h>0$ is the atomic spacing and $Z_k$ are independent Bernoulli random variables.  We choose $Z_k=1$ with probability $p$, which represents the presence of the impurity; with probability $1-p$ there is no impurity and $Z_k=0$.   The Laplace functionals for this model are easily evaluated: for any $\phi\in C_c(\R)$,
\begin{align}\label{555}
\bbE_{p,h}\Bigl\{ \exp\Bigl( - {\textstyle\int}\phi(x)\,d\mu(x) \Bigr)\Bigr\} &= \prod_{k\in\Z} \biggl( 1 + p [e^{-\phi(hk)} - 1] \biggr).
\end{align}

For any fixed $h>0$ and $p\in[0,1]$, the model described by \eqref{444} lies within the scope of Theorem~\ref{t:deterministic}.   However, the fundamental rationale for the concentrated nonlinearity is that the atomic scale is much smaller than the characteristic scale of the wave $\psi$; correspondingly, we must send $h\to 0$.  As $h\to0$, we send $p\to0$ in such a way that $p/h \to 1$; this ensures a limiting unit density of impurities.  In this regime, it is easy to verify that RHS\eqref{555} converges to RHS\eqref{E:pp}.  In this way, we see that the Poisson process also provides the natural description of low-density impurities in a crystal in the regime relevant to us.

\subsection*{Acknowledgements}
Rowan Killip was supported by NSF grants DMS-2154022 and DMS-2452346. Monica Vi\c san was supported by NSF grant DMS-2054194.  The authors are grateful to the anonymous referees for their comments and suggestions.

\section{Preliminaries}\label{s:prelim}

Throughout the paper, $\lesssim$ and $\gtrsim$ indicate bounds that hold up to a multiplicative constant whose exact value is of no consequence.  We abbreviate $A\lesssim B\lesssim A$ as $A\simeq B$.  We write $C(\R)$ for the space of bounded continuous functions on $\R$.

For all \(s>\frac12\) we have the Sobolev embedding \(H^s\hookrightarrow C(\R)\), as demonstrated by the Gagliardo--Nirenberg inequality
\eq{GN}{
\|f\|_{L^\infty}\lesssim\|f\|_{L^2}^{1-\frac1{2s}}\|f\|_{H^s}^{\frac1{2s}}.
}
As we consider finite-energy solutions, we will typically use the case $s=1$.  The principal properties of the space $H^1(\R)$ that we employ are the following:
\begin{align}
\|fg\|_{H^1} &\lesssim \|f\|_{L^\infty}\|g\|_{H^1} + \|f\|_{H^1}\|g\|_{L^\infty} \label{product} \\
\|fg\|_{H^1}&\lesssim \|f\|_{H^1}\|g\|_{H^1} \label{alg} \\
\|fg\|_{H^{-1}}&\lesssim \|f\|_{H^{-1}}\|g\|_{H^1} \label{dual product} .
\end{align}
The first estimate is a direct consequence of the product rule.  The bound \eqref{alg} expresses the algebra property of $H^1(\R)$ and follows from \eqref{product} and \eqref{GN}.  The final estimate \eqref{dual product} is implied by \eqref{alg} and the realization of $H^{-1}(\R)$ as the dual of $H^1(\R)$ through the usual $L^2$-paring.

We fix a non-negative function \(\chi\in \Test(-1,1)\) so that \(\chi_k(x) = \chi(x-k)\) satisfies \(\sum_{k\in \Z}\chi_k(x) = 1\). One may readily verify that
\[
\|f\|_{L^2}^2\simeq \sum_{k\in \Z}\|\chi_k f\|_{L^2}^2\qtq{and}\|f\|_{H^1}^2\simeq \sum_{k\in \Z}\|\chi_k f\|_{H^1}^2 .
\]
Using complex interpolation, it follows that
\eq{l2 Hs}{
\|f\|_{H^s}^2\simeq \sum_{k\in \Z}\|\chi_k f\|_{H^s}^2
}
for all \(0\leq s\leq 1\).

We now turn to the properties of the space \(\X\). Recalling the definitions \eqref{Nk} and \eqref{omega}, we have the following:

\begin{lem}\label{l:omega}
If \(\mu\) is a positive Borel measure on \(\R\) for which \(\cN_0(\mu)<\infty\), then \(\cN_k(\mu)<\infty\) for all \(k\in \Z\) and we have the estimates
\begin{alignat}{3}
|\cN_k(\mu)^2 - \cN_\ell(\mu)^2| &\leq |k - \ell|&&\qt{for all \(k,\ell\in \Z\)},\label{1-Lip Nk}
\\
\mu(I_\ell) \leq \cN_\ell(\mu) &\simeq \cN_k(\mu)&&\qt{whenever \(|k - \ell|\leq 3\)}.\label{compare mu Nk}
\end{alignat}

The corresponding weight \(w\) is a \(1\)-Lipschitz function on \(\R\) satisfying the estimate
\eq{growth}{
|w(x;\mu)|\leq \cN_0(\mu)^2+|x|,
}
and we have the equivalent norm
\eq{equiv norm}{
\|f\|_\X^2\simeq \sum_{k\in \Z}\cN_k(\mu)^2\|\chi_k f\|_{L^2}^2.
}
\end{lem}
\bpf
From the estimate
\[
\Bigl|\bigl[\mu(I_\ell)^2 - |k-\ell|\bigr] - \bigl[\mu(I_\ell)^2 - |k+1-\ell|\bigr]\Bigr|\leq 1,
\]
we see that if one of \(\cN_k(\mu)\), \(\cN_{k+1}(\mu)\) is finite, then the other is also finite and we may bound
\[
|\cN_k(\mu)^2 - \cN_{k+1}(\mu)^2|\leq 1.
\]
By induction on \(|k|\), if \(\cN_0(\mu)<\infty\) then we have \(\cN_k(\mu)<\infty\) for all \(k\in \Z\) and the estimate \eqref{1-Lip Nk} holds.

The inequality in \eqref{compare mu Nk} follows directly from the definition \eqref{Nk}. For the equivalence, we note that \eqref{1-Lip Nk} and the fact that \(\cN_k(\mu)^2\geq 4\) ensures that if \(|k - \ell|\leq 3\) then
\[
|\cN_k(\mu)^2 - \cN_\ell(\mu)^2|\leq 3 \leq \tfrac34\cN_k(\mu)^2.
\]

Next, by \eqref{1-Lip Nk} and the definition \eqref{omega}, we see that for all \(x,y\in \R\) we have
\eq{1-Lip omega}{
|w(x;\mu) - w(y;\mu)|\leq |x-y|,
}
and hence \(w\) is \(1\)-Lipschitz. The estimate \eqref{growth} then follows from the fact that \(w(0;\mu) = \cN_0(\mu)^2\). Moreover, if \(x\in \supp\chi_k\subseteq (k-1,k+1)\) then \eqref{1-Lip omega} yields
\[
|w(x;\mu) - \cN_k(\mu)^2| = |w(x;\mu) - w(k;\mu)|\leq 1 \leq \tfrac14  \cN_k(\mu)^2 .
\]
Using this, \eqref{equiv norm} follows from \eqref{l2 Hs}.
\epf

Next we discuss the relationship between $\cN_0(\mu)$ and the mollified measure $\mu^\epsilon$ defined in~\eqref{mueps}:

\begin{lem}
If \(\mu\) is a positive Borel measure on \(\R\) with \(\cN_0(\mu)<\infty\), then the function \(\dmue\in C^\infty(\R)\) and we have the estimate
\begin{align}
\bigl\|\chi_k\dmue\bigr\|_{L^1}&\lesssim \cN_k(\mu),  \label{mu-eps-L1}
\end{align}
uniformly for \(0<\epsilon\leq 1\), as well as the estimate
\eq{mu-eps-diff}{
\|\chi_k(d\mu^\eta - d\mu^{\epsilon})\|_{H^{-1}}\lesssim \sqrt\eta\, \cN_k(\mu),
}
uniformly for \(0\leq\epsilon<\eta\leq 1\).
\end{lem}

\bpf
As \(\cN_0(\mu)<\infty\), \(\mu\) is a locally finite Borel measure and so \(\dmue \in C^\infty(\R)\).

To prove the estimates \eqref{mu-eps-L1} and \eqref{mu-eps-diff}, we argue by duality. If \(0<\epsilon \leq 1\) and \(\phi\in L^\infty\) then the function \(\rho^\epsilon*(\chi_k \phi)\in C_c^\infty(k-2,k+2)\) satisfies
\[
\|\rho^\epsilon*(\chi_k\phi)\|_{L^\infty}\lesssim \|\phi\|_{L^\infty}.
\]
We may then apply \eqref{compare mu Nk} to obtain
\[
|\<\chi_k d\mu^\epsilon,\phi\>| = \left|\int_{k-2}^{k+2} \rho^\epsilon*(\chi_k\phi)\,d\mu\right| \lesssim \cN_k(\mu) \|\rho^\epsilon*(\chi_k\phi)\|_{L^\infty}\lesssim \cN_k(\mu)\|\phi\|_{L^\infty},
\]
which gives us \eqref{mu-eps-L1}.

For \eqref{mu-eps-diff} we take \(\phi\in H^1\) and apply Plancherel's Theorem followed by the Fundamental Theorem of Calculus to obtain
\begin{align*}
\bigl\<\chi_k\bigl(d\mu^\eta - d\mu^{\epsilon}\bigr),\phi\bigr\>&= \int_{k-2}^{k+2}\int_\R\bigl[\rho^\eta(x-y) - \rho^{\epsilon}(x-y)\bigr]\phi(x) \chi_k(x)\,dx\,d\mu(y)\\
&= \int_{k-2}^{k+2}\int_\R \int_{\epsilon}^\eta \overline{{\widehat\rho{\mkern3mu}}'(h\xi)}e^{iy\xi}\xi\,\widehat{(\phi\chi_k)}(\xi) \,dh\,d\xi\,d\mu(y).
\end{align*}
We then apply the Cauchy--Schwarz inequality, \eqref{compare mu Nk}, and then \eqref{alg} to bound
\[
\bigl|\bigl\<\chi_k\bigl(d\mu^\eta - d\mu^{\epsilon}\bigr),\phi\bigr\>\bigr|\lesssim \cN_k(\mu)\int_{\epsilon}^\eta \|{\widehat\rho{\mkern3mu}}'(h\xi)\|_{L_\xi^2}\|\xi\widehat{(\phi\chi_k)}\|_{L_\xi^2}\,dh\lesssim\sqrt\eta {\mkern2mu} \cN_k(\mu)\|\phi\|_{H^1}.\qedhere
\]
\epf

We conclude this section with several nonlinear estimates involving the doubly mollified density $V^\epsilon$ that was introduced in \eqref{V eps}:

\begin{lem}\label{l:measure ests}
Let \(\mu\) be a locally finite Borel measure on \(\R\).

If \(\frac23\leq s\leq 1\) and \(f_j\in H^s\cap \X\), then we have the estimate
\begin{align}\label{energy lot}
\int_\R \prod_{j=1}^4 |f_j| \,d\mu + \int_\R \prod_{j=1}^4 |f_j| \,d\mu^\epsilon + \int_\R \prod_{j=1}^4 |f_j| \,V^\epsilon\,dx \lesssim \prod_{j=1}^4\|f_j\|_\X^{\frac14}\|f_j\|_{H^s}^{\frac34},
\end{align}
uniformly for \(0<\epsilon\leq 1\), as well as
\begin{align}\label{energy lot cvgce}
\lim_{\epsilon \to 0}\int_\R \prod_{j=1}^4f_j\,V^\epsilon\,dx &=\int_\R  \prod_{j=1}^4f_j\,d\mu
	= \lim_{\epsilon \to 0}\int_\R \prod_{j=1}^4f_j\,d\mu^\epsilon .
\end{align}

If \(f_j\in H^1\cap \X\) and \(g\in C(\R)\), then
\begin{align}
\int_\R | f_1f_2\,g | \,d\mu + \bigl\|f_1f_2\,g\, \dmue\bigr\|_{L^1} &\lesssim\|g\|_{L^\infty}\prod_{j=1}^2\|f_j\|_\X^{\frac12}\|f_j\|_{H^1}^{\frac12},\label{measure nonlin 1}\\
\|f_1f_2\,g\,d\mu \|_{H^{-1}} &\lesssim \|g\|_{L^\infty}\prod_{j=1}^2\|f_j\|_\X^{\frac12}\|f_j\|_{H^1}^{\frac12},\label{measure nonlin 2}
\end{align}
uniformly for \(0<\epsilon\leq 1\).

Finally, if \(f\in H^1\cap \X\) and \(0<\epsilon\leq \eta\leq 1\), then
\eq{measure nonlin 3}{
\||f|^2f(d\mu^\eta - d\mu^{\epsilon}) \|_{H^{-1}} \lesssim \sqrt\eta\|f\|_\X\|f\|_{H^1}^2.
}
\end{lem}
\bpf
Given \(s>\frac12\) and \(f\in H^s\cap \X\), we apply the Gagliardo--Nirenberg inequality \eqref{GN} followed by H\"older's inequality and then the equivalent norms \eqref{l2 Hs} and \eqref{equiv norm} to estimate
\begin{equation}\label{measure Linf 1}
\begin{aligned}
\sum_{k\in \Z}\cN_k(\mu)^{2 - \frac1s} \|\chi_kf\|_{L^\infty}^2&\lesssim \sum_{k\in \Z}\cN_k(\mu)^{2-\frac1s} \|\chi_kf\|_{L^2}^{2-\frac1s}\|\chi_kf\|_{H^s}^{\frac1s}\\
&\lesssim \|f\|_\X^{2-\frac1s}\|f\|_{H^s}^{\frac1s}.
\end{aligned}
\end{equation}
Arguing similarly, if \(\frac23\leq s\leq 1\) we may bound
\begin{align}
\sum_{k\in \Z}\cN_k(\mu) \|\chi_kf\|_{L^\infty}^4\notag
&\lesssim  \sum_{k\in \Z}\cN_k(\mu) \|\chi_kf\|_{L^2}^{4-\frac2s} \|\chi_kf\|_{H^s}^{\frac2s}\notag\\
&\lesssim \sum_{k\in \Z}\cN_k(\mu) \|\chi_kf\|_{L^2} \|\chi_kf\|_{H^s}^3\notag\\
&\lesssim \left[\sum_{k\in \Z}\cN_k(\mu)^2 \|\chi_kf\|_{L^2}^2\right]^{\frac12}\left[\sum_{k\in \Z}\|\chi_kf\|_{H^s}^6\right]^{\frac12}\lesssim \|f\|_\X\|f\|_{H^s}^3.\label{measure Linf 2}
\end{align}

Turning to the estimate \eqref{energy lot}, we first note that by Sobolev embedding we have \(f_j\in C(\R)\).  In particular, for any \(k_j\in \Z\) we have \(\prod_{j=1}^4\chi_{k_j}f_j\in C_c(\R)\) and the integral \(\int_\R\prod_{j=1}^4 |\chi_{k_j}f_j|\,d\mu\) is well-defined. Using that \(\supp\chi_{k_j}\subseteq (k_j - 1,k_j+1)\), we may then apply \eqref{compare mu Nk} followed by H\"older's inequality and \eqref{measure Linf 2} to obtain
\begin{align*}
\sum_{k_1,k_2,k_3,k_4\in \Z}\int_\R\prod_{j=1}^4 |\chi_{k_j}f_j|\,d\mu &\lesssim \sum_{\substack{k_1,k_2,k_3,k_4\in \Z\\|k_i - k_\ell|\leq 1}}\prod_{j=1}^4 \cN_{k_j}(\mu)^{\frac14}\|\chi_{k_j} f_j\|_{L^\infty}\\
&\lesssim \prod_{j=1}^4 \left[\sum_{k_j\in \Z} \cN_{k_j}(\mu)\|\chi_{k_j}f_j\|_{L^\infty}^4\right]^{\frac14}\!\!\lesssim \prod_{j=1}^4\|f_j\|_\X^{\frac14}\|f_j\|_{H^s}^{\frac 34}.
\end{align*}
Applying the Fubini--Tonelli Theorem, we may then bound
\begin{align*}
\left|\int_\R\prod_{j=1}^4 f_j\,d\mu\right| &= \left|\sum_{k_1,k_2,k_3,k_4\in \Z}\int_\R\prod_{j=1}^4\chi_{k_j}f_j\,d\mu\right|\\&\leq \sum_{k_1,k_2,k_3,k_4\in \Z}\int_\R\prod_{j=1}^4 |\chi_{k_j}f_j|\,d\mu\lesssim \prod_{j=1}^4\|f_j\|_\X^{\frac14}\|f_j\|_{H^s}^{\frac 34},
\end{align*}
as claimed. For the second and third terms on \LHS{energy lot}, we first use \eqref{mu-eps-L1} and then \eqref{compare mu Nk} to estimate
\begin{align*}
\sum_{k_1,k_2,k_3,k_4\in\Z} \int_\R\prod_{j=1}^4 |\chi_{k_j}f_j| \,d\mu^\epsilon  +  \int_\R &\prod_{j=1}^4 |\chi_{k_j}f_j| V^\epsilon \,dx \\
&\lesssim \sum_{\substack{k_1,k_2,k_3,k_4,k_5\in \Z\\|k_i - k_\ell|\leq 1}}
	\bigl\|\chi_{k_5}\dmue\bigr\|_{L^1}\prod_{j=1}^4 \|\chi_{k_j} f_j\|_{L^\infty}\\
&\lesssim \sum_{\substack{k_1,k_2,k_3,k_4\in \Z\\|k_i - k_\ell|\leq 1}}\prod_{j=1}^4 \cN_{k_j}(\mu)^{\frac14} \|\chi_{k_j} f_j\|_{L^\infty},
\end{align*}
and the proof is then completed as for the first term.

Next we prove the first equality in \eqref{energy lot cvgce}; the second follows by the same argument.  We take \(K>1\) to be an integer and estimate
\begin{align*}
&\left|\int_\R \prod_{j=1}^4f_j\,V^\epsilon\,dx - \int_\R \prod_{j=1}^4f_j\,d\mu\right|\\
&\qquad\leq \sum_{\substack{k_1,k_2,k_3,k_4\in \Z\\|k_1|\leq K,\ |k_i - k_\ell|\leq 1}} \left|\int_\R \prod_{j=1}^4 \chi_{k_j}f_j\,V^\epsilon\,dx - \int_\R \prod_{j=1}^4 \chi_{k_j}f_j\,d\mu\right|\\
&\qquad\quad + \sum_{\substack{k_1,k_2,k_3,k_4\in \Z\\|k_1|> K,\ |k_i - k_\ell|\leq 1}} \left[\int_\R \prod_{j=1}^4 |\chi_{k_j}f_j||V^\epsilon|\,dx+ \int_\R \prod_{j=1}^4 |\chi_{k_j}f_j|\,d\mu\right].
\end{align*}
As \(\prod_{j=1}^4\chi_{k_j}f_j\in C_c(\R)\), for any fixed $k_1, \ldots, k_4\in \Z$ we have
\[
\lim_{\epsilon\to0}\int_\R \prod_{j=1}^4 \chi_{k_j}f_j\,V^\epsilon\,dx = \int_\R \prod_{j=1}^4 \chi_{k_j}f_j\,d\mu.
\]
Moreover, from \eqref{energy lot} we have
\begin{align*}
&\sup_{0<\epsilon \leq 1}\sum_{\substack{k_1,k_2,k_3,k_4\in \Z\\ |k_1|> K,|k_i - k_\ell|\leq 1}} \left[\int_\R \prod_{j=1}^4 |\chi_{k_j}f_j||V^\epsilon|\,dx+ \int_\R \prod_{j=1}^4 |\chi_{k_j}f_j|\,d\mu\right]\\
&\qquad\lesssim \sum_{\substack{k_1,k_2,k_3,k_4\in \Z\\|k_1|> K,|k_i - k_\ell|\leq 1}}\prod_{j=1}^4\|\chi_{k_j}f_j\|_\X^{\frac14}\|\chi_{k_j}f_j\|_{H^s}^{\frac34}\\
&\qquad\lesssim \prod_{j=1}^4\left[\sum_{|k_j|\geq K}\|\chi_{k_j}f_j\|_\X\|\chi_{k_j}f_j\|_{H^s}^3  \right]^{\frac14}.
\end{align*}
Taking $\varphi^{\frac{2}{K-1}}(x) = \varphi(\frac{2}{K-1}x)$ to be the smooth cutoff function discussed in \eqref{V eps}, we observe that \(1 - \varphi^{\frac 2{K-1}}\) is identically \(1\) for \(|x|\geq K-1\). As a consequence, we may replace \(f_j\) by \((1-\varphi^{\frac2{K-1}})f_j\) and then apply H\"older's inequality with \eqref{l2 Hs} and \eqref{equiv norm} to obtain
\begin{align*}
\prod_{j=1}^4\left[\sum_{|k_j|\geq K}\|\chi_{k_j}f_j\|_\X\|\chi_{k_j}f_j\|_{H^s}^3  \right]^{\frac14}\lesssim \prod_{j=1}^4\|(1-\varphi^{\frac2{K-1}})f_j\|_\X^{\frac14}\|(1-\varphi^{\frac2{K-1}})f_j\|_{H^s}^{\frac34}. 
\end{align*}
The estimate \eqref{energy lot cvgce} now follows from first taking \(\epsilon \to 0\) and then \(K\to\infty\).

The estimate \eqref{measure nonlin 1} is proved similarly to \eqref{energy lot}. For the first term on \LHS{measure nonlin 1} we apply \eqref{compare mu Nk} followed by the Cauchy--Schwarz inequality and \eqref{measure Linf 1} with $s=1$ to bound
\begin{align*}
\sum_{k_1,k_2\in \Z}\int_\R |\chi_{k_1}f_1| |\chi_{k_2}f_2| |g|\,d\mu
	&\lesssim \sum_{\substack{k_1,k_2\in \Z\\|k_1-k_2|\leq 1}}\sqrt{\cN_{k_1}\cN_{k_2}}\|\chi_{k_1} f_1\|_{L^\infty}\|\chi_{k_2}f_2\|_{L^\infty}\|g\|_{L^\infty}\\
&\lesssim \|g\|_{L^\infty}\prod_{j=1}^2\|f_j\|_\X^{\frac12}\|f_j\|_{H^1}^{\frac12}.
\end{align*}
For the second term on \LHS{measure nonlin 1} we use \eqref{mu-eps-L1} and \eqref{compare mu Nk} to estimate
\begin{align*}
\sum_{k_1,k_2}\int_\R |\chi_{k_1}f_1| |\chi_{k_2}f_2| |g| \bigl|\dmue\bigr|\,dx
	&\lesssim \!\!\!\sum_{\substack{k_1,k_2,k_3\\|k_j-k_\ell|\leq 1}}\!\!\!\bigl\|\chi_{k_3}\dmue\bigr\|_{L^1}\|\chi_{k_1} f_1\|_{L^\infty}
		\|\chi_{k_2}f_2\|_{L^\infty}\|g\|_{L^\infty}\\
&\lesssim\!\!\! \sum_{\substack{k_1,k_2\in \Z\\|k_1-k_2|\leq 1}}\!\!\!\!\sqrt{\cN_{k_1}\cN_{k_2}}\|\chi_{k_1} f_1\|_{L^\infty}\|\chi_{k_2}f_2\|_{L^\infty}\|g\|_{L^\infty},
\end{align*}
and then proceed as for the first term.

For \eqref{measure nonlin 2} we take \(\phi\in H^1\) and apply \eqref{measure nonlin 1} followed by \eqref{GN} to bound
\begin{align*}
\bigl|\<\phi,f_1f_2g\,d\mu\>\bigr|&\lesssim \|g\phi\|_{L^\infty}\prod_{j=1}^2\|f_j\|_\X^{\frac12}\|f_j\|_{H^1}^{\frac12}\\
&\lesssim \|g\|_{L^\infty}\|\phi\|_{H^1}\prod_{j=1}^2\|f_j\|_\X^{\frac12}\|f_j\|_{H^1}^{\frac12}.
\end{align*}
The estimate then follows by duality.

Finally, we apply the product estimates \eqref{dual product} and \eqref{product}, followed by \eqref{mu-eps-diff}, \eqref{compare mu Nk}, and finally \eqref{measure Linf 1} to estimate
\begin{align*}
&\||f|^2f(d\mu^\eta - d\mu^\epsilon)\|_{H^{-1}}\\
&\qquad\lesssim \sum_{\substack{k_1,k_2,k_3,k_4\in \Z\\|k_j - k_\ell|\leq 1}}\|\chi_{k_1}f\|_{L^\infty}\|\chi_{k_2} f\|_{L^\infty}\|\chi_{k_3}f\|_{H^1}\|\chi_{k_4}(d\mu^\eta - d\mu^\epsilon)\|_{H^{-1}}\\
&\qquad\lesssim \sum_{\substack{k_1,k_2\in \Z\\|k_1 - k_2|\leq 1}}\sqrt\eta\sqrt{\cN_{k_1}\cN_{k_2}}\|\chi_{k_1}f\|_{L^\infty}\|\chi_{k_2} f\|_{L^\infty}\|f\|_{H^1} \lesssim \sqrt\eta\,\|f\|_\X\|f\|_{H^1}^2,
\end{align*}
which yields \eqref{measure nonlin 3}.
\epf

\section{Deterministic well-posedness}\label{s:deterministic}

In this section we prove Theorem~\ref{t:deterministic}. We assume throughout that \(\mu\) is a fixed Borel measure for which \(\cN_0(\mu)<\infty\). We denote the space \(H^1\) endowed with the weak topology by \(H_w^1\).

We start by proving a Lipschitz estimate for the solution map of \eqref{NLS}, which yields uniqueness as an immediate corollary:

\begin{prop}\label{p:weak Lipschitz}
Suppose that for some \(T>0\) there are two solutions \(\psi,\phi\in C([-T,T];H_w^1\cap\X)\) of the Duhamel formulation \eqref{Duhamel NLS} of \eqref{NLS}. Then for all times \(|t|\leq T\) we have the estimate
\eq{Lipschitz}{
\|\psi(t) - \phi(t)\|_{H^{-1}}\leq C e^{CK^2(1 + K^2)|t|^2}\|\psi(0) - \phi(0)\|_{H^1},
}
where \(C>0\) is some absolute constant and
\eq{K defn}{
K = \sup_{|t|\leq T}\Bigl[\|\psi\|_{H^1}\|\psi\|_{\X} + \|\phi\|_{H^1}\|\phi\|_{\X}\Bigr].
}
\end{prop}
\bpf
Using the Duhamel formula \eqref{Duhamel NLS} we may write the difference as
\[
\psi(t) - \phi(t) = e^{it\p_x^2}\bigl[\psi(0) - \phi(0)\bigr] -2i \int_0^t e^{i(t-\tau)\p_x^2}\Bigl[\bigl(|\psi(\tau)|^2\psi(\tau) - |\phi(\tau)|^2\phi(\tau)\bigr)d\mu\Bigr]\,d\tau.
\]
To apply a Gronwall's inequality-type argument, we must estimate the difference \(\psi - \phi\) in the same norm on both the left and right hand sides of this identity. As we will see, the \(L^\infty\) norm is very convenient for this; correspondingly, we will first prove the Lipschitz-type bound in this norm; see \eqref{542}.  The \(H^{-1}\)-bound \eqref{Lipschitz} that we need later in our argument will be deduced as a consequence of \eqref{542}.

For the first term, we apply \eqref{GN} to bound
\eq{superworm}{
\bigl\|e^{it\p_x^2}\bigl[\psi(0) - \phi(0)\bigr] \bigr\|_{L^\infty}\lesssim \bigl\|e^{it\p_x^2}\bigl[\psi(0) - \phi(0)\bigr] \bigr\|_{H^1}\lesssim \|\psi(0) - \phi(0)\|_{H^1}.
}

For the second term, we take a test function \(f\in L^1\) so that for \(\tau\neq t\) we have \(e^{-i(t-\tau)\p_x^2}f\in C(\R)\) with the dispersive estimate
\eq{dispersive}{
\|e^{-i(t-\tau)\p_x^2}f\|_{L^\infty}\lesssim \tfrac1{\sqrt{|t-\tau|}}\|f\|_{L^1}.
}
Applying \eqref{measure nonlin 1}, for \(0\leq \tau<t\leq T\) we may then bound
\begin{align*}
&\left|\left\<f,e^{i(t-\tau)\p_x^2}\Bigl[\bigl(|\psi(\tau)|^2\psi(\tau) - |\phi(\tau)|^2\phi(\tau)\bigr)d\mu\Bigr]\right\>\right|\\
&\qquad = \left|\int_\R \overline{e^{-i(t-\tau)\p_x^2}f}\bigl(|\psi(\tau)|^2\psi(\tau) - |\phi(\tau)|^2\phi(\tau)\bigr)\,d\mu\right|\\
&\qquad\lesssim \Bigl[\|\psi(\tau)\|_{H^1}\|\psi(\tau)\|_\X + \|\phi(\tau)\|_{H^1}\|\phi(\tau)\|_\X\Bigr]\Bigl\|\bigl(\psi(\tau) - \phi(\tau)\bigr) \overline{e^{-i(t-\tau)\p_x^2}f}\Bigr\|_{L^\infty}\\
&\qquad\lesssim \tfrac1{\sqrt{t-\tau}}K\|\psi(\tau) - \phi(\tau)\|_{L^\infty}\|f\|_{L^1},
\end{align*}
where \(K\) is defined as in \eqref{K defn}. By duality, we then have
\eq{highway rat}{
\Bigl\|e^{i(t-\tau)\p_x^2}\Bigl[\bigl(|\psi(\tau)|^2\psi(\tau) - |\phi(\tau)|^2\phi(\tau)\bigr)d\mu\Bigr]\Bigr\|_{L^\infty}\lesssim \tfrac1{\sqrt{t-\tau}}K\|\psi(\tau) - \phi(\tau)\|_{L^\infty}.
}

Combining \eqref{superworm}, \eqref{highway rat}, and then applying H\"older's inequality, for \(0<t\leq T\) we may bound
\begin{align*}
\|\psi(t) - \phi(t)\|_{L^\infty}&\lesssim \|\psi(0) - \phi(0)\|_{H^1} + K\int_0^t \tfrac1{\sqrt{t-\tau}}\|\psi(\tau) - \phi(\tau)\|_{L^\infty}\,d\tau\notag\\
&\lesssim \|\psi(0) - \phi(0)\|_{H^1} + K\left[\int_0^t(t- \tau)^{-\frac23}\,d\tau\right]^{\frac34}\left[\int_0^t \|\psi(\tau) - \phi(\tau)\|_{L^\infty}^4\,d\tau\right]^{\frac14}\notag\\
&\lesssim \|\psi(0) - \phi(0)\|_{H^1} + \left[K^4 t\int_0^t \|\psi(\tau) - \phi(\tau)\|_{L^\infty}^4\,d\tau\right]^{\frac14}. 
\end{align*}
Taking the forth power, we arrive at
\begin{align}\label{Gronwall helper}
\|\psi(t) - \phi(t)\|_{L^\infty}^4 
&\lesssim \|\psi(0) - \phi(0)\|_{H^1}^4 + K^4 t\int_0^t \|\psi(\tau) - \phi(\tau)\|_{L^\infty}^4\,d\tau.
\end{align}
This is amenable to the classical form of Gronwall's inequality, yielding
\begin{equation}\label{542}
\|\psi(t) - \phi(t)\|_{L^\infty} \lesssim e^{CK^4t^2}\|\psi(0) - \phi(0)\|_{H^1}
\end{equation}
for some absolute constant $C>0$.

Finally, we apply \eqref{measure nonlin 2} to bound
\begin{align*}
\|\psi(t) - \phi(t)\|_{H^{-1}} &\lesssim \|\psi(0) - \phi(0)\|_{H^{-1}} + \!\int_0^t \bigl\|\bigl(|\psi(\tau)|^2\psi(\tau) - |\phi(\tau)|^2\phi(\tau)\bigr)d\mu\bigr\|_{H^{-1}}d\tau\\
&\lesssim \|\psi(0) - \phi(0)\|_{H^1} + K\!\int_0^t \|\psi(\tau) - \phi(\tau)\|_{L^\infty}\,d\tau\\
&\lesssim \Bigl[1 + Kt e^{CK^4t^2}\Bigr]\|\psi(0) - \phi(0)\|_{H^1},
\end{align*}
which yields \eqref{Lipschitz} for times \(0<t\leq T\). The corresponding estimate for times \(-T\leq t<0\) follows from a parallel argument.
\epf

Next, we turn to the problem of existence. Taking \(V^\epsilon\in\Test(\R)\) to be defined as in \eqref{V eps}, we consider solutions \(\psi^\epsilon\colon \R\times \R\to \C\) of the regularized equation,
\begin{gather}\label{regularized NLS}\tag{NLS\(_\epsilon\)}
i\psi_t^\epsilon = - \psi_{xx}^\epsilon + 2V^\epsilon|\psi^\epsilon|^2\psi^\epsilon.
\end{gather}

The proof of the following proposition can be found in~\cite{MR2002047}, with the exception of the final statement that follows from the argument of \cite{MR0641651}:
\begin{prop}\label{p:regularized}
For all \(\epsilon>0\) the equation \eqref{regularized NLS} is globally well-posed in \(H^1\).
More precisely, for all initial data \(\psi^\epsilon(0)\in H^1\) there exists a unique solution \(\psi^\epsilon\in C(\R;H^1)\) of the Duhamel formulation
\begin{equation*}
\psi^\epsilon(t) = e^{it\p_x^2}\psi^\epsilon(0) - 2i\int_0^t e^{i(t-\tau)\p_x^2}\Bigl[V^\epsilon|\psi^\epsilon(\tau)|^2\psi^\epsilon(\tau)\Bigr]\,d\tau
\end{equation*}
of \eqref{regularized NLS}. This solution conserves both the mass and the corresponding energy,
\[
E^\epsilon[\psi^\epsilon] = \int_\R \tfrac12|\psi_x^\epsilon|^2 + \tfrac12|\psi^\epsilon|^4\,V^\epsilon\,dx.
\]
Moreover, for all times \(T>0\) the solution map \(H^1\ni\psi^\epsilon(0)\mapsto \psi^\epsilon\in  C([-T,T];H^1)\) is locally Lipschitz. Finally, if \(\psi^\epsilon(0)\in \Schwartz\) then \(\psi^\epsilon\in C^\infty(\R; \Schwartz)\).
\end{prop}

We wish to prove that $\psi^\epsilon(t)$ converges as $\epsilon\to0$.  A key first step is to obtain bounds that are uniform for \(\epsilon\in(0,1]\):

\begin{lem}\label{l:AP bds}
If \(\psi_0\in H^1\cap \X\) and \(\psi^\epsilon\in C(\R;H^1)\) is the solution of \eqref{regularized NLS} with initial data \(\psi^\epsilon(0) = \psi_0\), then \(\psi^\epsilon\in C(\R;\X)\) and there exists a constant \(C>0\) so that for all \(T>0\) we have the estimates
\begin{align}
\sup_{t\in \R}\|\psi^\epsilon(t)\|_{H^1}^2 &\leq C \Bigl[1 + \|\psi_0\|_\X\|\psi_0\|_{H^1}\Bigr]\|\psi_0\|_{H^1}^2 ,\label{AP bound} \\
\sup_{|t|\leq T}\|\psi^\epsilon(t)\|_{\X}^2 &\leq  \|\psi_0\|_\X^2 + CT\sup_{t\in \R}\|\psi^\epsilon(t)\|_{H^1}^2 ,\label{Main AP}\\
\sup_{|t|\leq T} \|(1 - \varphi^{\lambda})\psi^\epsilon(t)\|_{L^2}^2 &\leq \|(1 - \varphi^{\lambda})\psi_0\|_{L^2}^2 
	+ C\lambda T\sup_{t\in \R}\|\psi^\epsilon(t)\|_{H^1}^2 ,\label{tight L2}\\
\sup_{|t|\leq T}\|(1 - \varphi^{\lambda})\psi^\epsilon(t)\|_{\X}^2 &\leq \|(1 - \varphi^{\lambda})\psi_0\|_\X^2\label{tight}\\
&\quad + CT\<\cN_0^2\lambda\>\sup_{|t|\leq T}\Bigl[\|(1 - \varphi^{\lambda})\psi^\epsilon(t)\|_{L^2}\|\psi^{\epsilon}(t)\|_{H^1}\Bigr],  \notag
\end{align}
uniformly for \(0<\epsilon,\lambda\leq 1\).
\end{lem}

\bpf
Combining an approximation argument with the well-posedness provided by Proposition~\ref{p:regularized}, we see that it suffices to prove these bounds for \(\psi_0\in \Schwartz\).  As noted in Proposition~\ref{p:regularized}, this ensures that that the solution \(\psi^\epsilon\in C^\infty(\R;\Schwartz)\).

By conservation of mass and energy,
\begin{align*}
\sup_{0<\epsilon\leq 1}\sup_{t\in \R}\|\psi^\epsilon(t)\|_{H^1}^2\lesssim \sup_{0<\epsilon\leq 1}\sup_{t\in\R}\Bigl[M[\psi^\epsilon(t)] + E^\epsilon[\psi^\epsilon(t)]\Bigr] &= M[\psi_0] + \sup_{0<\epsilon\leq 1}E^\epsilon[\psi_0].
\end{align*}
Thus by \eqref{energy lot}, we have 
\begin{align*}
\sup_{0<\epsilon\leq 1}\sup_{t\in \R}\|\psi^\epsilon(t)\|_{H^1}^2  &\lesssim \Bigl[1 + \|\psi_0\|_\X\|\psi_0\|_{H^1}\Bigr]\|\psi_0\|_{H^1}^2,
\end{align*}
which proves \eqref{AP bound}. 

As \(\psi^\epsilon(t)\in \Schwartz\), the estimate \eqref{growth} ensures that \(\|\psi^\epsilon(t)\|_\X<\infty\) for all \(t\in \R\). We may then integrate by parts to obtain 
\begin{equation}
\tfrac d{dt}\|\psi^\epsilon\|_\X^2 = -2\Im \int_\R \bigl(\overline{\psi^\epsilon} \psi_x^\epsilon\bigr)_x\,w\,dx = 2\Im \int_\R \overline{\psi^\epsilon} \psi_x^\epsilon\,w_x\,dx,
\end{equation}
and so, by \eqref{1-Lip omega},
\[
\sup_{|t|\leq T} \|\psi^\epsilon(t)\|_\X^2 \leq \|\psi_0\|_\X^2 + 2T\sup_{|t|\leq T}\|\psi^\epsilon(t)\|_{L^2}\|\psi^\epsilon(t)\|_{H^1},
\]
which proves \eqref{Main AP}.   The estimates \eqref{tight L2} and \eqref{tight} follow from a parallel argument, with \(w\) replaced by \((1 - \varphi^{\lambda})^2\) and \((1 - \varphi^{\lambda})^2w\), respectively.  In the second instance, we also use that
\[
\bigl| [ (1-\varphi^{\lambda})^2w ]_x (x)\bigr| \lesssim (1-\varphi^{\lambda}(x))\bigl[ 1 + \lambda |\varphi_x(\lambda x)| |w(x)| \bigr]
	\lesssim\<\cN_0^2\lambda\> (1-\varphi^{\lambda}(x)),
\]
which follows from \eqref{growth} and \eqref{1-Lip omega}.

Given initial data \(\psi_0\in H^1\cap \X\), the estimate \eqref{Main AP} ensures that the solution \(\psi^\epsilon\in C(\R;H^1) \cap L^\infty([-T,T];\X)\). To complete the proof we have to verify that $\psi^\epsilon\in C([-T,T];\X)$.  To this end, we take \(t,\tau\in[-T,T]\), \(0<\lambda\leq 1\), and use \eqref{growth} to bound
\begin{align*}
\|\psi^\epsilon(t) - \psi^\epsilon(\tau)\|_\X &\leq \|\varphi^{\lambda}[\psi^\epsilon(t) - \psi^\epsilon(\tau)]\|_\X + 2\sup_{|t|\leq T}\|(1 - \varphi^{\lambda})\psi^\epsilon(t)\|_\X\\
&\lesssim \bigl(\cN_0 + \tfrac1{\sqrt{\lambda}}\bigr)\|\psi^\epsilon(t) - \psi^\epsilon(\tau)\|_{L^2} + \sup_{|t|\leq T}\|(1 - \varphi^{\lambda})\psi^\epsilon(t)\|_\X.
\end{align*}
As the map \(\R\ni t\mapsto \psi^\epsilon(t)\) is $L^2$-continuous, we deduce that
\[
\limsup_{t\to\tau}\|\psi^\epsilon(t) - \psi^\epsilon(\tau)\|_\X\lesssim \sup_{|t|\leq T}\|(1 - \varphi^{\lambda})\psi^\epsilon(t)\|_\X.
\]
Employing \eqref{AP bound} and \eqref{tight L2} in  \eqref{tight}, we send \(\lambda\to0\) and obtain
\[
\lim_{t\to\tau}\|\psi^\epsilon(t) - \psi^\epsilon(\tau)\|_\X = 0,
\]
which proves that $\psi^\epsilon\in C([-T,T];\X)$.
\epf

Using these uniform estimates for \(\psi^\epsilon\), we now prove existence of local-in-time solutions to \eqref{NLS}:

\begin{prop}\label{p:local existence}
Given \(\psi_0\in H^1\cap \X\),  there exists $T>0$ and a solution \(\psi\in C([-T,T];H_w^1\cap H^s\cap \X)\), for every $0\leq s<1$, to the Duhamel formulation \eqref{Duhamel NLS} of \eqref{NLS}.  
The solution $\psi(t)$ conserves mass and satisfies the energy inequality
\eq{energy ineq}{
E\bigl[\psi(t)\bigr]\leq E\bigl[\psi_0\bigr],
}
for all $|t|\leq T$. Moreover, if $\psi^\epsilon$ is a family of solutions to \eqref{regularized NLS} with $\psi^\epsilon(0) \to \psi_0$ in $H^1\cap\X$, then 
\begin{equation}\label{local convergence}
\lim_{\epsilon\to0}\ \sup_{|t|\leq T} \ \bigl\|\psi^\epsilon(t) - \psi(t)\bigr\|_{H^s\cap\X} =0 \qtq{for each} 0\leq s<1.
\end{equation}
\end{prop}

\begin{proof}
We will construct $\psi(t)$ as a limit of solutions $\psi^\epsilon(t)$ to \eqref{regularized NLS}.  For this purpose, it would suffice to consider solutions $\psi^\epsilon$ to \eqref{regularized NLS} with $\psi^\epsilon(0)=\psi_0$, rather than with a sequence of initial data $\psi^\epsilon(0)\to\psi_0$ in $H^1\cap \X$.  However, the extra generality we consider here will be useful when proving \eqref{local convergence}.  It is no loss of generality to demand that
\begin{equation}\label{not too big}
\sup_{0<\epsilon\leq 1}\|\psi^\epsilon(0) \|_{H^1} \leq \|\psi_0\|_{H^1} + 1 \qtq{and } \sup_{0<\epsilon\leq 1}\|\psi^\epsilon(0) \|_{\X} \leq \|\psi_0\|_{\X} + 1.
\end{equation}

Given \(T>0\), we define the control quantities
\[
A = \sup_{0<\epsilon\leq 1}\sup_{|t|\leq T}\|\psi^\epsilon(t)\|_{H^1}\qtq{and}B = \sup_{0<\epsilon\leq 1}\sup_{|t|\leq T}\|\psi^\epsilon(t)\|_\X.
\]
Recalling \eqref{AP bound} and \eqref{not too big}, we have 
\begin{align}
A^2 &\lesssim \bigl[1 + \|\psi_0\|_\X\bigr]\bigl[1 + \|\psi_0\|_{H^1}\bigr]^3,\label{A bound}
\end{align}
and using \eqref{Main AP} as well, we have
\eq{B bound}{
B^2\leq \sup_{0<\epsilon\leq 1}\|\psi^\epsilon(0)\|_\X^2 + CT A^2
\lesssim [1+T]\bigl[1 + \|\psi_0\|_\X\bigr]^2\bigl[1 + \|\psi_0\|_{H^1}\bigr]^3.
}

Taking \(0<\epsilon\leq \eta\leq 1\), we use the Duhamel formula \eqref{Duhamel NLS} to write
\begin{align*}
\psi^\eta(t) - \psi^{\epsilon}(t) &= e^{it\partial_x^2}[\psi^\eta(0) - \psi^{\epsilon}(0)]  +  I + \II + \III,
\end{align*}
where
\begin{align*}
I&:= - 2i\int_0^t e^{i(t-\tau)\p_x^2}\Bigl[\bigl(\varphi^\eta - \varphi^{\epsilon}\bigr) |\psi^\eta(\tau)|^2\psi^\eta(\tau) d\mu^\eta \Bigr]\,d\tau, \\
\II&:= - 2i\int_0^t e^{i(t-\tau)\p_x^2}\Bigl[\varphi^{\epsilon} |\psi^\eta(\tau)|^2\psi^\eta(\tau) \bigl(d\mu^\eta - d\mu^{\epsilon}\bigr) \Bigr]\,d\tau, \\
\III &:= - 2i\int_0^t e^{i(t-\tau)\p_x^2}\Bigl[\varphi^{\epsilon}\bigl[ |\psi^\eta(\tau)|^2\psi^\eta(\tau) - |\psi^{\epsilon}(\tau)|^2\psi^{\epsilon}(\tau)\bigr]d\mu^{\epsilon} \Bigr]\,d\tau.
\end{align*}

Given dyadic \(N\geq 1\) we define the Littlewood--Paley projection
\[
\widehat{P_{\leq N}f}(\xi) = \varphi\bigl(\tfrac\xi N\bigr)\widehat f(\xi),
\]
where $\varphi$ is the same bump function \(\varphi\in \Test(-2,2)\) employed in \eqref{V eps}.  We also use the notation \(P_{>N} = 1 - P_{\leq N}\). We begin by bounding the low frequency part of each of the terms $I,\II,\III$ in both \(L^\infty\) and \(H^s\) for $0\leq s<1$.

Using the dispersive estimate \eqref{dispersive} and then \eqref{measure nonlin 1}, we bound the first term by
\begin{align*}
\sup_{|t|\leq T}\|P_{\leq N}I\|_{L^\infty} &\lesssim \sup_{|t|\leq T}\left|\int_0^t \tfrac1{\sqrt{|t-\tau|}}\bigl\|(\varphi^\eta - \varphi^{\epsilon}\bigr)
	|\psi^\eta(\tau)|^2\psi^\eta(\tau) \dmut \bigr\|_{L^1}\,d\tau\right|\\
&\lesssim  \sqrt T AB\sup_{|t|\leq T}\|(\varphi^\eta - \varphi^{\epsilon})\psi^\eta(t)\|_{L^\infty}.
\end{align*}
Replacing the dispersive estimate by Bernstein's inequality, we similarly obtain
\begin{align*}
\sup_{|t|\leq T}\|P_{\leq N}I\|_{H^s} &\lesssim N^{s+\frac12}T \sup_{|t|\leq T}
	\bigl\|(\varphi^\eta - \varphi^{\epsilon}\bigr)  |\psi^\eta(t)|^2\psi^\eta(t) \dmut\bigr\|_{L^1}\\
&\lesssim N^{s+\frac12}T AB\sup_{|t|\leq T}\|(\varphi^\eta - \varphi^{\epsilon})\psi^\eta(t)\|_{L^\infty},
\end{align*}
for any $0\leq s<1$.

The estimates for the third term are essentially identical. Applying the dispersive estimate \eqref{dispersive} followed by \eqref{measure nonlin 1} we get
\begin{align*}
\sup_{|t|\leq T}\|P_{\leq N}\III\|_{L^\infty}\lesssim  \sqrt T AB\sup_{|t|\leq T}\|\psi^\eta(t) - \psi^{\epsilon}(t)\|_{L^\infty},
\end{align*}
while replacing the dispersive estimate by Bernstein's inequality yields
\[
\sup_{|t|\leq T}\|P_{\leq N}\III\|_{H^s} \lesssim N^{s+\frac12} T AB\sup_{|t|\leq T}\|\psi^\eta(t) - \psi^{\epsilon}(t)\|_{L^\infty}.
\]

For the remaining term, we first apply \eqref{measure nonlin 3} to bound
\begin{align*}
\sup_{|t|\leq T}\|P_{\leq N}\II\|_{H^{-1}}&\lesssim T\sup_{|t|\leq T}\bigl\| |\psi^\eta(t)|^2\psi^\eta(t)  \bigl(d\mu^\eta - d\mu^{\epsilon}\bigr) \bigr\|_{H^{-1}}\lesssim T A^2B\sqrt \eta.
\end{align*}
We may then use Bernstein's inequality to get
\begin{align*}
\sup_{|t|\leq T}\|P_{\leq N}\II\|_{L^\infty} &\lesssim N^{\frac32} T A^2B\sqrt \eta,\\
\sup_{|t|\leq T}\|P_{\leq N}\II\|_{H^s} &\lesssim N^{s+1} T A^2B\sqrt \eta.
\end{align*}

We now combine our estimates at low frequencies with Bernstein's inequality at high frequencies and \eqref{GN} to bound
\begin{align}
\sup_{|t|\leq T} \|\psi^\eta(t) & - \psi^{\epsilon}(t)\|_{L^\infty}\notag\\
&\quad\lesssim \|\psi^\eta(0) - \psi^{\epsilon}(0)\|_{H^1} +  \sup_{|t|\leq T}\|P_{>N}\psi^\eta(t)\|_{L^\infty} + \sup_{|t|\leq T}\|P_{>N}\psi^{\epsilon}(t)\|_{L^\infty} \notag\\
&\quad\qquad + \sup_{|t|\leq T}\|P_{\leq N}I\|_{L^\infty} + \sup_{|t|\leq T}\|P_{\leq N}\II\|_{L^\infty} + \sup_{|t|\leq T}\|P_{\leq N}\III\|_{L^\infty}\notag\\
&\quad\lesssim \|\psi^\eta(0) - \psi^{\epsilon}(0)\|_{H^1} +  N^{-\frac12}A +\sqrt T AB\sup_{|t|\leq T}\|(\varphi^\eta - \varphi^{\epsilon})\psi^\eta(t)\|_{L^\infty}\label{inf diff}\\
&\quad\qquad  + N^{\frac32} T A^2B\sqrt\eta + \sqrt T AB\sup_{|t|\leq T}\|\psi^\eta(t) - \psi^{\epsilon}(t)\|_{L^\infty}.\notag
\end{align}
We now choose \(T>0\) sufficiently small (depending only on \(\psi_0\)) to ensure that
\[
\sqrt T AB\ll1,
\]
and hence the final term on \RHS{inf diff} can be absorbed into the left-hand side. From \eqref{GN} and \eqref{tight L2}, we may bound the third term on \RHS{inf diff} as follows:
\begin{align*}
&\sup_{|t|\leq T}\|(\varphi^\eta - \varphi^{\epsilon})\psi^\eta(t)\|_{L^\infty}^4\\
&\qquad\lesssim \sup_{|t|\leq T}\biggl\{\Bigl[\|(1 - \varphi^{\eta})\psi^\eta(t)\|_{L^2}^2 + \|(1 - \varphi^{\epsilon})\psi^\eta(t)\|_{L^2}^2\Bigr]
	\|(\varphi^\eta - \varphi^{\epsilon})\psi^\eta(t)\|_{H^1}^2\biggr\}\\
&\qquad\lesssim \Bigl[\|(1 - \varphi^{\eta})\psi_0\|_{L^2}^2 + \|(1 - \varphi^{\epsilon})\psi_0\|_{L^2}^2 + \eta T A^2\Bigr]A^2.
\end{align*}
This converges to zero as \(\eta,\epsilon\to0\). As a consequence, we may pass to the limit in \eqref{inf diff} to obtain
\[
\limsup_{\eta,\epsilon\to0}\sup_{|t|\leq T}\|\psi^\eta(t) - \psi^{\epsilon}(t)\|_{L^\infty}\lesssim N^{-\frac12}A.
\]
Taking the limit as \(N\to\infty\), we see that \(\psi^\epsilon\) converges in \(C([-T,T];L^\infty)\) as $\epsilon\to0$.  Let $\psi$ denote this limit.

Arguing as in \eqref{inf diff} and choosing \(T\) so that $\sqrt{T} AB\ll1$, but using our \(H^s\) low frequency bounds in place of those in \(L^\infty\),  we get
\begin{align*}
\sup_{|t|\leq T}\|\psi^\eta(t) - \psi^{\epsilon}(t)\|_{H^s} &\lesssim \|\psi^\eta(0) - \psi^{\epsilon}(0)\|_{H^s} + N^{s-1}A \\
&\quad + N^{s+\frac12} \sqrt T\sup_{|t|\leq T}
	\|(\varphi^\eta - \varphi^{\epsilon})\psi^\eta(t)\|_{L^\infty}\\
&\quad + N^{s+1} \sqrt T A\sqrt\eta + N^{s+\frac12}\sqrt T\sup_{|t|\leq T}\|\psi^\eta(t) - \psi^{\epsilon}(t)\|_{L^\infty}.
\end{align*}
Passing to the limit as \(\eta,\epsilon\to 0\)  and then \(N\to\infty\) now shows that \(\psi^\epsilon\to\psi\) in \(C([-T,T];H^s)\) for all $0\leq s <1$.

As the solutions \(\psi^\epsilon\) conserve mass, we may pass to the limit as \(\epsilon \to0\) to obtain
\[
\|\psi(t)\|_{L^2} = \|\psi_0\|_{L^2}\qt{for all \(|t|\leq T\)}.
\]

In view of \eqref{growth} and the \(C([-T,T];L^2)\) convergence already proved,
\[
\limsup_{\eta,\epsilon\to 0}\sup_{|t|\leq T} \bigl\| \varphi^{\lambda}(\psi^\eta(t) - \psi^{\epsilon}(t)) \bigr\|_{\X}^2
\lesssim \Bigl(\cN_0^2 + \tfrac1{\lambda}\Bigr)  \limsup_{\eta,\epsilon\to 0} \bigl\| \psi^\eta - \psi^{\epsilon} \bigr\|_{C_tL^2}^2
=0
\]
for any fixed $\lambda >0$.  Thus, in order to demonstrate that \(\psi^\epsilon \to\psi\) in \(C([-T,T];\X)\), it suffices to show that
\begin{equation}\label{777}
\limsup_{\lambda\to 0}\  \sup_{0< \epsilon\leq 1} \ \sup_{|t|\leq T} \bigl\| [1-\varphi^{\lambda}] \psi^{\epsilon}(t) \bigr\|_{\X}^2 =0.
\end{equation}
To this end, we first employ \eqref{tight L2} to obtain
\begin{equation*}
\limsup_{\lambda\to 0}\  \sup_{0< \epsilon\leq 1} \ \sup_{|t|\leq T} \bigl\| [1-\varphi^{\lambda}] \psi^{\epsilon}(t) \bigr\|_{L^2}^2
\leq \limsup_{\lambda\to 0}\  \sup_{0< \epsilon\leq 1} \bigl\| [1-\varphi^{\lambda}] \psi^{\epsilon}(0) \bigr\|_{L^2}^2 =0.
\end{equation*}
Using this in \eqref{tight}, we may then conclude that
\begin{equation*}
\text{LHS\eqref{777}} \leq \limsup_{\lambda\to 0}\  \sup_{0\leq \epsilon\leq 1} \bigl\| [1-\varphi^{\lambda}] \psi^{\epsilon}(0) \bigr\|_{\X}^2 =0.
\end{equation*}

Weak compactness ensures that for each time \(t\in[-T,T]\) we have weak convergence \(\psi^\epsilon(t)\rightharpoonup \psi(t)\) in \(H^1\). As \(\psi^\epsilon\in C([-T,T];H^1)\), it follows that \(\psi\in C([-T,T];H_w^1)\).

Fixing $\frac23\leq s <1$, we use \eqref{energy lot} to obtain
\begin{align*}
&\left|\int_\R |\psi^\epsilon|^4\,V^\epsilon\,dx - \int_\R |\psi|^4\,d\mu\right|\\
&\qquad\leq \left|\int_\R \Bigl[|\psi^\epsilon|^4 - |\psi|^4\Bigr]\,V^\epsilon\,dx\right| + \left|\int_\R |\psi|^4\,V^\epsilon\,dx - \int_\R |\psi|^4\,d\mu\right|\\
&\qquad\lesssim\Bigl[\|\psi^\epsilon\|_{H^s\cap \X}^3 + \|\psi\|_{H^s\cap \X}^3\Bigr] \|\psi^\epsilon - \psi\|_{H^s\cap \X} + \left|\int_\R |\psi|^4\,V^\epsilon\,dx - \int_\R |\psi|^4\,d\mu\right|,
\end{align*}
which then converges to zero as \(\epsilon \to0\) by \eqref{energy lot cvgce}. As a consequence, for all \(|t|\leq T\) we have
\begin{align}\label{enerquality}
E\bigl[\psi(t)\bigr] &\leq  \liminf_{\epsilon\to 0}E^\epsilon[\psi^\epsilon(t)] = \liminf_{\epsilon\to 0}E^\epsilon\bigl[\psi_0\bigr] = E\bigl[\psi_0\bigr],
\end{align}
giving us the energy inequality \eqref{energy ineq}.

To complete the proof, we have one final task: verifying that \(\psi\) is a solution to \eqref{Duhamel NLS}.  For any \(\phi\in \Schwartz\), we have
\begin{equation*}
\|e^{-i(t-\tau)\p_x^2}\phi\|_{H^s} = \|\phi\|_{H^s}\qtq{and}\||x| e^{-i(t-\tau)\p_x^2}\phi\|_{L^2}\lesssim \||x|\phi\|_{L^2} + |t - \tau|\|\phi\|_{H^1};
\end{equation*}
the second estimate here is a consequence of the identity
\[
x e^{-i(t-\tau)\p_x^2} \phi = e^{-i(t-\tau)\p_x^2}[x\phi + 2i(t-\tau)\p_x\phi].
\]
In this way, \eqref{growth} yields
\[
\|e^{-i(t-\tau)\p_x^2}\phi\|_{H^s\cap \X}\lesssim_\phi \cN_0 + \sqrt{|t - \tau|}.
\]
We may then argue as in the proof of \eqref{energy ineq}, using \eqref{energy lot} and \eqref{energy lot cvgce} to show that for all \(|t|\leq T\) we have
\begin{align*}
&\lim_{\epsilon\to 0}\left\<\phi,\int_0^t e^{i(t-\tau)\p_x^2}\Bigl[V^\epsilon|\psi^\epsilon(\tau)|^2\psi^\epsilon(\tau)\Bigr]\,d\tau\right\>\\
&\qquad= \lim_{\epsilon \to 0}\int_0^t\int_\R \overline{e^{-i(t-\tau)\p_x^2}\phi} \, V^\epsilon|\psi^\epsilon(\tau)|^2\psi^\epsilon(\tau)\,dx\,d\tau\\
&\qquad= \int_0^t\int_\R \overline{e^{-i(t-\tau)\p_x^2}\phi} \, |\psi(\tau)|^2\psi(\tau)\,d\mu\,d\tau\\
&\qquad= \left\<\phi,\int_0^t e^{i(t-\tau)\p_x^2}\Bigl[|\psi(\tau)|^2\psi(\tau)d\mu\Bigr]\,d\tau\right\>,
\end{align*}
so \(\psi\) satisfies \eqref{Duhamel NLS} in the sense of (tempered) distributions.
\epf

To complete the proof of the existence of solutions to \eqref{NLS} as stated in Theorem~\ref{t:deterministic}, we must extend the local-in-time solution constructed in Proposition~\ref{p:local existence} to a global-in-time solution.  We then combine the uniqueness result from Proposition~\ref{p:weak Lipschitz} with the energy inequality \eqref{energy ineq} to improve the time continuity properties of our solution.  Following this strategy, we will prove the following:

\begin{prop}\label{p:global existence}
Given \(\psi_0\in H^1\cap \X\) there exists a unique strong solution \(\psi\in C(\R;H^1\cap\X)\) of \eqref{NLS}. This solution conserves both the mass and the energy. Moreover, if $\psi^\epsilon$ denote the solutions to \eqref{regularized NLS} with initial data $\psi_0$, then
\begin{equation}\label{global convergence}
\lim_{\epsilon\to0}\ \sup_{|t|\leq T} \ \bigl\|\psi^\epsilon(t) - \psi(t)\bigr\|_{H^1\cap\X} =0 \qtq{for each} 0<T<\infty.
\end{equation} 
\end{prop}

\bpf
We first note that uniqueness follows from Proposition~\ref{p:weak Lipschitz}.

Turning to the problem of global existence, let \(I\subseteq (0,\infty)\) be the set of times \(T>0\) for which there exists a solution \(\psi\in C([-T,T];H_w^1\cap H^s\cap \X)\) for all  \(0\leq s<1\) of  \eqref{NLS} with \(\psi(0) = \psi_0\), which conserves mass and satisfies the energy inequality \eqref{energy ineq}.

By definition, \(I\) is (relatively) closed. Proposition~\ref{p:local existence} shows that \(I\) is non-empty. Moreover, if \(T\in I\) then we may apply Proposition~\ref{p:local existence} with initial data \(\psi(\pm T)\in H^1\cap \X\) to show that \(I\) is open. As a consequence, \(I = (0,\infty)\), yielding global existence.

Let us now suppose for a contradiction that there exists some time \(T\neq 0\) for which
\[
E\bigl[\psi(T)\bigr]<E\bigl[\psi_0\bigr].
\]
We may then define
\[
\phi(t,x) = \psi(T + t,x)
\]
to obtain a solution \(\phi\in C(\R;H_w^1\cap H^s\cap \X)\) for all $0\leq s<1$ of \eqref{NLS} for which
\[
E\bigl[\phi(-T)\bigr]>E\bigl[\phi(0)\bigr],
\]
contradicting the fact that the unique solution of \eqref{NLS} with initial data \(\phi(0)\) satisfies the energy inequality \eqref{energy ineq}. As a consequence, we must have
\[
E\bigl[\psi(t)\bigr] = E\bigl[\psi_0\bigr]
\]
for all times \(t\in \R\).

From \eqref{energy lot} and the fact that \(\psi\in C(\R;H^s\cap \X)\) for $\frac23\leq s<1$, the map
\[
t\mapsto\int_\R |\psi(t)|^4\,d\mu
\]
is continuous on \(\R\). Combining this with the conservation of energy we see that for all \(t_0\in \R\) we have
\[
\lim_{t\to t_0}\|\psi(t)\|_{H^1}^2 = \lim_{t\to t_0}\left[ M[\psi(t)] + 2E[\psi(t)] - \int_\R |\psi(t)|^4\,d\mu\right] = \|\psi(t_0)\|_{H^1}^2.
\]
Consequently, the map \(\R\ni t\mapsto \|\psi(t)\|_{H^1}\) is continuous. As \(\psi\in C(\R;H_w^1)\), by the Radon--Riesz theorem we then have \(\psi\in C(\R;H^1)\).

It remains only to prove \eqref{global convergence}.  From Proposition~\ref{p:local existence}, we obtain convergence in a short time interval $[-T,T]$, but only in the weaker $C([-T,T];H^s\cap\X)$ norm.  Nevertheless, this allows us to deduce weak convergence $\psi^\epsilon(t_\epsilon)\rightharpoonup\psi(t)$ in $H^1$, whenever $t_\epsilon\to t\in[-T,T]$.  We will demonstrate how to use conservation of energy to deduce convergence in $C([-T,T]; H^1)$ . This combination of Proposition~\ref{p:local existence} and upgraded convergence can then be iterated to cover any compact time interval.

If convergence of $\psi^\epsilon$ to $\psi$ in $C([-T,T];H^1)$ were to fail, then (passing to a subsequence) we could find $t_\epsilon\to t\in[-T,T]$ such that
\begin{align}\label{star}
\inf_\epsilon\bigl\|\psi^\epsilon(t_\epsilon) - \psi(t)\bigr\|_{H^1} >0.
\end{align}
It is important here that we have already showed that \(\psi\in C(\R;H^1)\).

By energy conservation and \eqref{energy lot cvgce}, we have
\begin{equation*}
E[\psi(t)]=E[\psi_0]=\lim_{\epsilon \to 0} E^\epsilon[\psi_0] = \lim_{\epsilon \to 0} E^\epsilon[\psi^\epsilon(t_\epsilon)].
\end{equation*}
Combining this with the convergence of the potential energies (see the proof of \eqref{enerquality}), we deduce that the $H^1$ norms of $\psi^\epsilon(t_\epsilon)$ converge to that of $\psi(t)$.  By the Radon--Riesz theorem and the weak convergence noted above, this implies norm convergence, contradicting \eqref{star}.
\epf

To complete the proof of Theorem~\ref{t:deterministic}, it remains to prove that the solution map is continuous.

\begin{prop}\label{p:cont}
If \(\psi_{n,0}\to \psi_0\) is a convergent sequence of \(H^1\cap \X\) initial data with corresponding solutions \(\psi_n,\psi\in C(\R;H^1\cap \X)\), then for any time \(T>0\) we have
\[
\lim_{n\to\infty}\sup_{|t|\leq T}\|\psi_n(t) - \psi(t)\|_{H^1\cap \X} = 0.
\]
\end{prop}
\bpf
Our starting point is Proposition~\ref{p:weak Lipschitz}, which yields
\[
\lim_{n\to\infty}\sup_{|t|\leq T}\|\psi_n(t) - \psi(t)\|_{H^{-1}} = 0.
\]

Next, we upgrade this to convergence in \(C([-T,T];H^{\frac23})\). From the conservation of mass and energy followed by \eqref{energy lot}, we may bound
\[
M[\psi_n(t)] + E\bigl[\psi_n(t)\bigr] = M[\psi_{n,0}] + E\bigl[\psi_{n,0}\bigr]\lesssim \Bigl[1 + \|\psi_{n,0}\|_\X\|\psi_{n,0}\|_{L^2}\Bigr]\|\psi_{n,0}\|_{H^1}^2,
\]
from which we obtain the uniform bound
\eq{uniform H1 along seq}{
\sup_{n\geq 1}\sup_{|t|\leq T}\|\psi_n(t)\|_{H^1}^2\leq \sup_{n\geq 1}\sup_{|t|\leq T}\Bigl[M[\psi_n(t)] + 2 E[\psi_n(t)]\Bigr] <\infty.
}
As a consequence, for dyadic \(N\geq 1\) we may bound
\begin{align*}
\sup_{|t|\leq T}\|\psi_n(t) - \psi(t)\|_{H^{\frac23}} & \leq \sup_{|t|\leq T}\|P_{\leq N}(\psi_n(t) - \psi(t))\|_{H^{\frac23}} + \sup_{|t|\leq T}\|P_{>N}\psi_n(t)\|_{H^{\frac23}}\\
&\quad + \sup_{|t|\leq T}\|P_{>N}\psi(t)\|_{H^{\frac23}} \\
&\lesssim N^{\frac53}\sup_{|t|\leq T}\|\psi_n(t) - \psi(t)\|_{H^{-1}} + N^{-\frac13}\sup_{n\geq 1}\sup_{|t|\leq T}\|\psi_n(t)\|_{H^1}\\
&\quad + N^{-\frac13}\sup_{|t|\leq T}\|\psi(t)\|_{H^1}.
\end{align*}
By first taking \(n\to\infty\) and then \(N\to\infty\), we see that \(\psi_n\to \psi\) in \(C([-T,T];H^{\frac23})\).

We recall that a compact subset \(Q\subseteq H^{-1}\) is bounded and equicontinuous, and consequently
\[
\lim_{N\to\infty}\sup_{\phi\in Q}\|P_{>N}\phi\|_{H^{-1}} = 0;
\]
see, e.g., \cite{MR3990604,MR0801333}. Mimicking our argument for convergence in \(C([-T,T];H^{\frac23})\), we see that for any compact subset \(Q\subseteq H^{-1}\) we have
\begin{align*}
&\sup_{|t|\leq T}\sup_{\phi\in Q}|\<\phi,\psi_n(t) - \psi(t)\>|\\
&\qquad\leq \sup_{|t|\leq T}\sup_{\phi\in Q}|\<\phi,P_{\leq N}(\psi_n(t) - \psi(t))\>| + \sup_{|t|\leq T}\sup_{\phi\in Q}|\<\phi,P_{>N}\psi_n(t)\>|\\
&\qquad \quad + \sup_{|t|\leq T}\sup_{\phi\in Q}|\<\phi,P_{>N}\psi(t)\>|\\
&\qquad\lesssim N^2 \sup_{\phi\in Q}\|\phi\|_{H^{-1}}\sup_{|t|\leq T}\|\psi_n(t) - \psi(t)\|_{H^{-1}}\\
&\qquad\quad + \sup_{\phi\in Q}\|P_{>N}\phi\|_{H^{-1}}\sup_{n\geq 1}\sup_{|t|\leq T}\|\psi_n(t)\|_{H^1} + \sup_{\phi\in Q}\|P_{>N}\phi\|_{H^{-1}}\sup_{|t|\leq T}\|\psi(t)\|_{H^1},
\end{align*}
so once again taking \(n\to\infty\) and then \(N\to\infty\) gives us
\eq{weak continuity in H1}{
\lim_{n\to\infty}\;\sup_{|t|\leq T}\;\sup_{\phi\in Q}\;|\<\phi,\psi_n(t) - \psi(t)\>| = 0.
}
We will return to this result later when proving $C([-T,T];H^1)$ convergence.

Turning to convergence in \(C([-T,T];\X)\), we first note that as \(\psi_{n,0}\to \psi_0\) in \(\X\) we have
\begin{equation}\label{923}
\lim_{\lambda\to0}\;\sup_{n\geq 1}\;\|(1 - \varphi^{\lambda})\psi_{n,0}\|_\X = 0\qtq{and}
\lim_{\lambda\to0}\;\sup_{n\geq 1}\;\|(1 - \varphi^{\lambda})\psi_{n,0}\|_{L^2} = 0.
\end{equation}

In view of the convergence \eqref{global convergence}, using \eqref{tight L2}, \eqref{tight}, \eqref{uniform H1 along seq}, and \eqref{923}, we find that
\begin{equation*}
\lim_{\lambda\to 0}\  \sup_{|t|\leq T}\ \Bigl[ \|(1 - \varphi^{\lambda})\psi(t)\|_{L^2}^2 + \sup_n\|(1 - \varphi^{\lambda})\psi_n(t)\|_{L^2}^2 \Bigr]= 0
\end{equation*}
and
\begin{equation}\label{933}
\lim_{\lambda\to 0}\ \sup_{|t|\leq T}\ \Bigl[ \|(1 - \varphi^{\lambda})\psi(t)\|_\X^2 + \sup_n \|(1 - \varphi^{\lambda})\psi_n(t)\|_{\X}^2\Bigr] =0.
\end{equation}

We are now ready to address convergence in \(C([-T,T];\X)\).  Clearly
\begin{align*}
&\|\psi_n(t) - \psi(t)\|_\X \leq \|\varphi^{\lambda}[\psi_n(t) - \psi(t)]\|_\X + \|(1 - \varphi^{\lambda})\psi_n(t)\|_\X
	+ \|(1 - \varphi^{\lambda})\psi(t)\|_\X
\end{align*}
holds for all $\lambda>0$ and $t\in[-T,T]$.  Using \eqref{growth} we may bound
\begin{align*}
\sup_{|t|\leq T} \|\varphi^{\lambda}[\psi_n(t) - \psi(t)]\|_\X& \lesssim \bigl(\cN_0 + \tfrac1{\sqrt{\lambda}}\bigr)\sup_{|t|\leq T}\|\psi_n(t) - \psi(t)\|_{L^2}.
\end{align*}
Recalling \eqref{933} and that we have already proved convergence in \(C([-T,T];H^{\frac23})\), we may send \(n\to\infty\) and then \(\lambda\to0\), we deduce that \(\psi_n\to \psi\) in \(C([-T,T];\X)\).

As we have now shown that \(\psi_n\to \psi\) in \(C([-T,T];H^{\frac23}\cap \X)\), \eqref{energy lot} gives us
\[
\lim_{n\to\infty}\sup_{|t|\leq T}\left|\int_\R|\psi_n(t)|^4\,d\mu - \int_\R|\psi(t)|^4\,d\mu\right| = 0.
\]
Conservation of mass and energy then yields
\begin{align*}
\sup_{|t|\leq T}\Bigl|\|\psi_n(t)\|_{H^1}^2 - \|\psi(t)\|_{H^1}^2\Bigr| & \leq \Bigl| M\bigl[\psi_{n,0}\bigr] - M[\psi_0]\Bigr] + 2\Bigl|E\bigl[\psi_{n,0}\bigr] - E\bigl[\psi_0\bigr] \Bigr|\\
&\quad + \sup_{|t|\leq T}\left|\int_\R|\psi_n(t)|^4\,d\mu - \int_\R|\psi(t)|^4\,d\mu\right|,
\end{align*}
which converges to zero as \(n\to\infty\).

Finally, as \(\psi\in C(\R;H^1)\), the set
\[
Q = \bigl\{(1 - \p_x^2)\psi(t):|t|\leq T\bigr\}\subseteq H^{-1}
\]
is compact. We may then use \eqref{weak continuity in H1} with the convergence of norms to see that
\begin{align*}
\sup_{|t|\leq T}\|\psi_n(t) - \psi(t)\|_{H^1}^2 &\leq \sup_{|t|\leq T}\Bigl|\|\psi_n(t)\|_{H^1}^2 - \|\psi(t)\|_{H^1}^2\Bigr|\\
&\quad + 2\sup_{|t|\leq T}\sup_{\phi\in Q}\Bigl|\<\phi,\psi_n(t) - \psi(t)\>\Bigr|
\end{align*}
converges to zero as \(n\to\infty\).
\epf

We now turn our attention to a deterministic version of Theorem~\ref{t:demol}.  This theorem concerns solutions $\psi_\epsilon$ to $(\text{NLS}_{\mu^\epsilon})$ with  $\mu^\epsilon$ defined by \eqref{mueps}.  Specifically, we prove an analogue of \eqref{global convergence} for the solutions $\psi_\epsilon$ to
\begin{align}\label{NLSmue}
i\partial_t \psi_\epsilon = -\partial_x^2 \psi_\epsilon + 2|\psi_\epsilon|^2\psi_\epsilon \,d\mu^\epsilon
	\qtq{with} \psi_\epsilon(0) = \psi_0 \in \X \cap H^1.
\end{align}

\begin{cor}\label{C:demol}
Suppose $\cN_0(\mu)<\infty$ and $\psi_0\in \X \cap H^1$.  For each $\epsilon\in(0,1]$, there is a unique global solution $\psi_\epsilon$ to \eqref{NLSmue}.  Moreover,
\begin{equation}\label{muec}
\lim_{\epsilon\to0}\ \sup_{|t|\leq T} \ \bigl\|\psi_\epsilon(t) - \psi(t)\bigr\|_{H^1\cap\X} =0 \qtq{for each} 0<T<\infty.
\end{equation}
\end{cor}

\begin{proof}
For $\epsilon\in(0,1]$, the definition \eqref{mueps} of $\mu^\epsilon$ guarantees that
$$
\mu^\epsilon( I_k) \leq \mu( I_{k-1})+ \mu( I_k)+ \mu( I_{k+1}) \qtq{and} \mu( I_k) \leq \mu^\epsilon( I_{k-1})+ \mu^\epsilon( I_k)+ \mu^\epsilon( I_{k+1})
$$
and consequently,
\begin{align}\label{999}
\cN_k(\mu) \simeq \cN_k(\mu^\epsilon) \qtq{and} w(x;\mu)\simeq w(x;\mu^\epsilon) \quad\text{uniformly for $\epsilon\in(0,1]$.}
\end{align}
This demonstrates that the norms on $\X$ and $\Xe$ are equivalent, uniformly for $\epsilon\in(0,1]$.
In particular, $\psi_0\in \X$ implies that $\psi_0 \in L^2_{\mu^\epsilon}$.  Thus, Theorem~\ref{t:deterministic} guarantees that there is a unique global solution $\psi_\epsilon$ to \eqref{muec}; moreover, this solution is $H^1\cap\X$-continuous and conserves both the mass and the energy functional
$$
E_\epsilon[\psi] := \int \tfrac12 |\psi'(x)|^2\,dx + \int \tfrac12|\psi(x)|^4 \,d\mu^\epsilon(x).
$$

Using mass and energy conservation, as well as  \eqref{global convergence} and  \eqref{Main AP}, we also see that
\begin{align*}
\sup_{0<\epsilon\leq 1}\sup_{|t|\leq T} \|\psi_\epsilon(t)\|_{H^1\cap \X} < \infty \qtq{for any} 0<T<\infty.
\end{align*}
The existence of the solutions $\psi_\epsilon$ and the fact that they obey such bounds allow us to repeat the proof of Proposition~\ref{p:local existence} replacing $\varphi^\eta\equiv\varphi^\epsilon\equiv 1$ (but retaining $\varphi^\lambda$).   In this way, we find that $\psi_\epsilon(t)$ are Cauchy in $C([-T,T];\X\cap H^{2/3})$ and that the limit is the unique solution $\psi(t)$ to \eqref{NLS}.  Although this convergence is local in time, once we demonstrate $C([-T,T];H^1)$ convergence, this argument can be iterated as in the proof of Proposition~\ref{p:global existence}.

The proof of Proposition~\ref{p:global existence} also contains the key idea for upgrading convergence, namely, energy conservation and the Radon--Riesz argument.
Writing
\begin{align*}
\biggl| \int |\psi|^4 \,d\mu - \int |\psi_\epsilon|^4 \,d\mu^\epsilon\biggr| 
	&\leq \biggr|\int |\psi|^4 \,[d\mu - d\mu^\epsilon] \biggr| + \int \bigl|  |\psi|^4 - |\psi_\epsilon|^4 \bigr| \,d\mu^\epsilon
\end{align*}
and then applying \eqref{energy lot cvgce} and \eqref{measure nonlin 1} (where we note that \(d\mu^\epsilon = \dmue\,dx\)), we deduce that
\begin{align*}
\lim_{\epsilon\to 0} \int |\psi_\epsilon(t)|^4 \,d\mu^\epsilon = \int |\psi(t)|^4 \,d\mu \qtq{for all} t\in[-T,T].
\end{align*}
In this way, energy conservation shows that $\|\psi_\epsilon(t)\|_{H^1} \to \|\psi(t)\|_{H^1}$ and the Radon--Riesz theorem allows us to upgrade convergence in the weak topology to norm convergence.
\end{proof}

\section{Proof of Theorem~\ref{t:main}}\label{s:main proof}

To apply Theorem~\ref{t:deterministic} to prove Theorem~\ref{t:main}, we require one additional lemma:
\begin{lem}\label{l:Nk finite}
Let \(\mu\) be distributed according to the Poisson process \eqref{E:pp}. Then
\eq{expect N0}{
\bbE\,\cN_0^2\lesssim 1,
}
and \(\cN_k<\infty\) for all \(k\in \Z\) almost surely. Moreover, if \(f\in L^2\) then
\eq{expect X}{
\bbE\|f\|_\X^{2p}\lesssim \|f\|_{L^2}^{2p}
}
for any $1\leq p<\infty$; in particular, \(f\in \X\) almost surely.
\end{lem}
\bpf
Fixing a choice of \(r>4\), we apply Young's inequality to bound
\[
\mu(I_\ell)^2|k-\ell|^{\frac r2-1}\leq \tfrac2r\bigl(\tfrac{r-2}r\bigr)^{\frac{r-2}2}\mu(I_\ell)^r + |k-\ell|^{\frac r2};
\]
in particular, if \(k\neq \ell\) then
\[
\mu(I_\ell)^2 - |k-\ell|\lesssim\frac{\mu(I_\ell)^r}{\<k-\ell\>^{\frac r2-1}}
\]
and so, for any $1\leq p<\infty$,
\[
\cN_k^{2p}\lesssim 1 + \sum_{\ell\in \Z}\frac{\mu(I_\ell)^{pr}}{\<k-\ell\>^{p(\frac r2-1)}}.
\]

We recall that the \(\mu(I_\ell)\) are independent Poisson random variables with mean \(\bbE\mu(I_\ell) = 1\). As a consequence,
\[
\sup_{\ell\in \Z}\bbE\mu(I_\ell)^{pr}\lesssim 1,
\]
and hence
\eq{upper Nk}{
\bbE\bigl( \cN_k^{2p} \bigr)\lesssim 1 + \sum_{\ell\in\Z} \frac{\bbE\mu(I_\ell)^{pr}}{\<k-\ell\>^{p(\frac r2-1)}}\lesssim1,
}
which gives us \eqref{expect N0}. Consequently, \(\cN_0^2<\infty\) almost surely and the estimate \eqref{1-Lip Nk} then ensures that almost surely we have \(\cN_k^2\leq \cN_0^2 + |k|<\infty\) for all \(k\in \Z\).

Similarly, if \(f\in L^2\) then we may apply \eqref{equiv norm}, Minkowski's inequality, \eqref{upper Nk}, and then \eqref{l2 Hs} to bound
\begin{align*}
\Bigl\{\bbE\|f\|_\X^{2p} \Bigr\}^{\frac1p}\simeq \Bigl\{ \bbE\Bigl[ \sum_{k\in \Z} \cN_k^2\|\chi_k f\|_{L^2}^2 \Bigr]^p \Bigr\}^{\frac1p}
&\lesssim \sum_{k\in \Z}\Bigl\{ \bbE\Bigl[ \cN_k^{2p}\|\chi_k f\|_{L^2}^{2p}\Bigr] \Bigr\}^{\frac1p}\\
&\lesssim  \sum_{k\in \Z} \|\chi_k f\|_{L^2}^2\simeq \|f\|_{L^2}^2,
\end{align*}
from which we obtain \eqref{expect X}.
\epf

We are now in a position to complete the proofs of our main theorems:

\bpf[Proof of Theorem~\ref{t:main}] Given \(\psi_0\in H^1\), Lemma~\ref{l:Nk finite} ensures that we almost surely have \(\cN_0<\infty\) and \(\psi_0\in \X\). We may then apply Proposition~\ref{p:global existence} to deduce that almost surely there exists a unique strong solution \(\psi\in C(\R;H^1\cap \X)\) of \eqref{NLS} that conserves mass and energy.

It remains to prove our continuity statement \eqref{continuity}.  We begin by fixing $2\leq p <\infty$, $T>0$, and a sequence of initial data \(\psi_{n,0}\) that converges to \(\psi_0\) in \(H^1\).  It suffices to show that every subsequence of $\psi_{n,0}$ admits a further subsequence that satisfies \eqref{continuity}. To keep the notation unburdened by further subscripts, all subsequences of $\psi_{n,0}$ will still be denoted $\psi_{n,0}$.  We start by observing that any subsequence of $\psi_{n,0}$ admits a further subsequence that satisfies two additional properties:
\begin{equation}\label{ID subs}
\sup_n \| \psi_{n,0} \|_{H^1} \lesssim 1 + \| \psi_{0} \|_{H^1} \qtq{and}  \sum_n \| \psi_{n,0} - \psi_0 \|_{H^1} \lesssim 1 + \| \psi_{0} \|_{H^1} .
\end{equation}

To prove that \eqref{continuity} holds for sequences satisfying \eqref{ID subs}, we will combine Proposition~\ref{p:cont} with the dominated convergence theorem.  Our main task is verifying the existence of a suitable dominating function.  We choose
\begin{equation}\label{F_defn}
F_T(\mu) := \sup_{n} \sup_{|t|\leq T} \Bigl[ \|\psi(t)\|_{H^1}^2 + \|\psi_n(t)\|_{H^1}^2 \Bigr]^{\frac p2} ,
\end{equation}
where $\psi(t)$ and $\psi_n(t)$ denote the solutions to \eqref{NLS} with initial data $\psi_0$ and $\psi_{n,0}$, respectively.  In light of the conservation of mass and energy,
\begin{align*}
F_T(\mu) &\leq \sup_{n} \biggl[ \|\psi_0\|_{H^1}^2 + \|\psi_{n,0}\|_{H^1}^2  + \int |\psi_0|^4 + |\psi_{n,0}|^4 \,d\mu \biggr]^{\frac p2}  \\
&\lesssim_p \|\psi_0\|_{H^1}^p + \biggl[ \int |\psi_0|^4 \,d\mu\biggr]^{\frac p2} \!+ \sum_n \biggl\{\| \psi_{n,0} - \psi_0 \|_{H^1}^p +  \biggl[  \int |\psi_{n,0}-\psi_0|^4 \,d\mu \biggr]^{\frac p2} \biggr\}.
\end{align*}
Employing \eqref{energy lot}, then \eqref{expect X}, and finally \eqref{ID subs}, we find that
\begin{equation}\label{df}
\bbE\{ F_T(\mu) \} \lesssim_p 1 + \|\psi_0\|_{H^1}^p + \|\psi_0\|_{H^1}^{2p}.
\end{equation}
This shows that our dominating function has finite expectation.

By Jensen's inequality, followed by \eqref{expect X} and \eqref{ID subs}, we also have
$$
\bbE \Bigl\{ \sum_n \|\psi_{n,0}-\psi_0\|_\X \Bigr\} \leq \sum_n \sqrt{ \bbE \|\psi_{n,0}-\psi_0 \smash[b]{\|_\X^2 }} < \infty,
$$
which shows that $\psi_{n,0}\to\psi_0$ in $\X$ almost surely.  By hypothesis, we know that $\psi_{n,0}\to \psi_0$ in $H^1$.  These two observations allow us to apply Proposition~\ref{p:cont} and so deduce that
\begin{equation}\label{as conv}
\lim_{n\to\infty}\sup_{|t|\leq T}\|\psi_n(t) - \psi(t)\|_{H^1} = 0 \quad\text{almost surely.} 
\end{equation}
Thus, we see that \eqref{continuity} follows from the dominated convergence theorem.
\epf

\begin{proof}[Proof of Theorem~\ref{t:demol}]
Lemma~\ref{l:Nk finite} shows that both \(\cN_0<\infty\) and \(\psi_0\in \X\) almost surely.  Thus, Corollary~\ref{C:demol} shows that \eqref{muec} holds almost surely.  To deduce $L^p(d\bbP)$ convergence, we employ the dominated convergence theorem as in the proof of Theorem~\ref{t:main}.  As a dominating function, we choose
\begin{equation*}
F_T(\mu) :=  \sup_{|t|\leq T} \; \Bigl[ \|\psi(t)\|_{H^1}^2 + \sup_{0<\epsilon\leq 1} \,\|\psi_\epsilon(t)\|_{H^1}^2 \Bigr]^{\frac p2}.
\end{equation*}
Using the conservation of mass and energy, as well as \eqref{energy lot} and \eqref{999}, we have
\begin{equation*}
F_T(\mu) \lesssim \Bigl[  M[\psi_0] + E[\psi_0] + \sup_{0<\epsilon\leq 1} E_\epsilon[\psi_0] \Bigr]^{\frac p2}
	\lesssim \|\psi_0\|_{H^1}^p + \|\psi_0\|_{H^1}^{3p/2}\|\psi_0\|_{\X}^{p/2} .
\end{equation*}
Thus, using \eqref{expect X}, we see that $\bbE\{ F_T\}<\infty$ and the proof is complete.
\end{proof}

\bibliographystyle{habbrv}
\bibliography{refs}

\end{document}